\documentclass{article}
\usepackage{color}
\usepackage{amsmath,amssymb,amsfonts,amsthm,enumerate}
\usepackage{mathrsfs,accents}
\usepackage[T1]{fontenc}
\usepackage[active]{srcltx}
\usepackage{epsfig}
\usepackage{graphicx}

\def\R{\mathbb R}

\def\N{\mathbb N}

\def\div{\mbox{div}}

\def\<{\langle}
\def\>{\rangle}

\numberwithin{equation}{section}

\newtheorem{theorem}{Theorem}

\newtheorem{lemma}[theorem]{Lemma}
\newtheorem{proposition}[theorem]{Proposition}

\newtheorem{definition}[theorem]{Definition\rm}
\newtheorem{remark}[theorem]{Remark}

\numberwithin{theorem}{section}

\newcommand{\omM}{\omega_M }
\newcommand{\omS}{\omega_{{\mathscr S}}}

\title{Strong Sard Conjecture and regularity of singular minimizing geodesics for analytic sub-Riemannian structures in dimension 3}

\author{A.~Belotto da Silva\thanks{Universit\'e Aix-Marseille Institut de Math\'{e}matiques de Marseille, UMR CNRS 7373,
Centre de Math\'{e}matiques et Informatique, 39, rue F. Joliot Curie, 13013 Marseille, France ({\tt andre-ricardo.belotto-da-silva@univ-amu.fr})}
\and
A.~Figalli\thanks{ETH Z\"urich, Mathematics Department, R\"amistrasse 101, 8092 Z\"urich, Switzerland ({\tt alessio.figalli@math.ethz.ch})}
\and
A.~Parusi\'nski\thanks{Universit\'e C\^ote d'Azur, CNRS, Labo.\ J.-A.\ Dieudonn\'e, UMR CNRS 7351, Parc Valrose, 06108 Nice Cedex 02, France ({\tt adam.parusinski@unice.fr})}
\and
L.~Rifford\thanks{Universit\'e C\^ote d'Azur, CNRS, Inria, Labo.\ J.-A.\ Dieudonn\'e,  UMR CNRS 7351, Parc Valrose, 06108 Nice Cedex 02, France ({\tt ludovic.rifford@math.cnrs.fr})}  
}

\date{}

\makeindex

\begin{document}

\maketitle

\begin{abstract}
In this paper we prove the strong Sard conjecture for sub-Riemannian structures on 3-dimensional analytic manifolds. More precisely, given a totally nonholonomic analytic distribution of rank 2 on a 3-dimensional analytic manifold, we investigate the size of the set of points that can be reached by singular horizontal paths starting from a given point and prove that it has Hausdorff dimension at most $1$. In fact, provided that the lengths of the singular curves under consideration are bounded with respect to a given complete Riemannian metric, we demonstrate that such a set is a semianalytic curve. As a consequence, combining our techniques with recent developments on the regularity of sub-Riemannian minimizing geodesics, we prove that minimizing sub-Riemannian geodesics in 3-dimensional analytic manifolds are always of class $C^1$, and actually are analytic outside of a finite set of points. 
\end{abstract}


\section{Introduction}\label{Introduction}

Let $M$ be a smooth connected manifold of dimension $n\geq 3$ and $\Delta$ a  {\it totally nonholonomic distribution} of rank $m<n$ on $M$, that is, a smooth subbundle of $TM$ of dimension $m$ generated locally by $m$ smooth vector fields $X^1, \ldots,X^m$ satisfying the H\"ormander condition
\begin{eqnarray*}
\mbox{Lie} \Bigl\{ X^1, \ldots,X^m \Bigr\} (x) = T_xM \qquad \forall\,x \in \mathcal{V}.
\end{eqnarray*}
By the Chow-Rashevsky Theorem, $M$ is horizontally path-connected with respect to $\Delta$. In other words, for every pair of points $x,y \in M$ there is a {\it horizontal path}  $\gamma : [0,1] \rightarrow M$ connecting them, i.e., an absolutely continuous curve $\gamma : [0,1] \rightarrow M$ with derivative in $L^2$, satisfying
\begin{eqnarray*}
\dot{\gamma}(t) \in \Delta_{\gamma(t)} \qquad  \mbox{for a.e. } t \in [0,1],\qquad \gamma(0)=x,\quad \gamma(1)=y.
\end{eqnarray*}
For every $x\in M$, denote by $\Omega_{\Delta}^{x}$ the set of horizontal paths on $[0,1]$ starting from $x$ equipped with the Sobolev $W^{1,2}$-topology. The Sard conjecture is concerned with the set of points that can be reached from a given $x\in M$ by the so-called singular curves in $\Omega_{\Delta}^x$. In order to state precisely the Sard conjecture it is convenient to identify the horizontal paths with the trajectories of a control system. For further details on notions and results of sub-Riemannian geometry\footnote{Actually, sub-Riemannian geometry is concerned with the study of structures of the form $(\Delta,g)$, called sub-Riemannian structures or SR structures, where $\Delta$ is a totally nonholonomic distribution on $M$ and $g$ is a metric over $\Delta$. We do not need to consider a metric over $\Delta$ to state the Sard conjectures investigated in this paper but we shall need a metric later, both for the second part of our first theorem and for our second theorem.} given in the introduction, we refer the reader to Bella\"iche's monograph \cite{bellaiche96}, or to the books by Montgomery \cite{montgomery02}, by Agrachev, Barilari and Boscain \cite{abb}, or by the fourth author \cite{riffordbook}. \\

Given a distribution $\Delta$ as above, it can be shown that there exist an integer $k \in [m,m(n+1)]$ and a family of $k$ smooth vector fields $\mathcal F=\{X^1,\ldots,X^k\}$ such that
\begin{eqnarray*}
\Delta_x = \mbox{Span} \Bigl\{ X^1(x), \ldots, X^k(x) \Bigr\}\qquad \forall\,x \in M. 
\end{eqnarray*}
For every $x\in M$, the set of {\it controls} $u=(u_1,\ldots,u_k) \in L^2([0,1],\R^k)$ for which the solution $\mathbf{x}(\cdot)=\mathbf{x}(\cdot \,;x,u)$ to the Cauchy problem 
\begin{eqnarray*}
\dot{\mathbf{x}}(t) = \sum_{i=1}^k u_i(t) X^i(\mathbf{x}(t)) \quad \mbox{for a.e. } t \in [0,1] \quad \mbox{and} \quad  \mathbf{x}(0)=x, 
\end{eqnarray*}
exists over $[0,1]$ is a nonempty open set $\mathcal{U}^x\subset L^2([0,1],\R^k)$. By construction, any solution $\mathbf{x}(\cdot; x,u) : [0,1] \rightarrow M$ with $u\in \mathcal{U}^x$ is a horizontal path in $\Omega_{\Delta}^{x}$. Moreover, by definition, any path $\gamma \in \Omega_{\Delta}^{x}$ is equal to $\mathbf{x}(\cdot \,;x,u)$   for some $u\in \mathcal{U}^x$ (which is not necessarily unique). Given a point $x\in M$, the {\em End-Point Mapping} from $x$ (associated with $\mathcal{F}$ in time $1$) is defined as 
\begin{eqnarray*}
\label{system}
 \begin{array}{rcl}
\mbox{E}^{x} :   \mathcal{U}^{x} & \longrightarrow & M \\
u  & \longmapsto &  \mathbf{x}(1;x,u),
\end{array}
\end{eqnarray*}
and it is of class $C^{\infty}$ on $\mathcal{U}^x$ equipped with the $L^2$-topology. A control $u \in \mathcal{U}^{x}\subset  L^2([0,1], \R^k)$ is called {\it singular} (with respect to $E^u$)  if the linear mapping 
$$
D_u\mbox{E}^{x} \, : \, L^2 \left([0,1],\R^k\right) \, \longrightarrow T_{\scriptsize\mbox{E}^{x}(u)} M
$$
 is not onto, that is, if $\mbox{E}^{x}$ is not a submersion at $u$. Then, a  horizontal path $\gamma \in \Omega_{\Delta}^{x}$ is called \emph{singular} if it coincides with $\mathbf{x}(\cdot\,; x,u)$  for some singular control $u \in \mathcal{U}^x$.  It is worth noticing that the property of singularity of a horizontal path  is independent of the choice of the family $X^1, \ldots, X^k$ and of the control $u$ which is chosen to parametrize the path. For every $x\in M$, we denote by $\mathcal{S}^{x}_{\Delta}$ the set of singular horizontal paths starting at $x$ and we set 
 $$
\mathcal{X}^x_{\Delta} := \Bigl\{ \gamma (1) \, \vert \, \gamma \in \mathcal{S}^{x}_{\Delta}\Bigr\} \subset M.
$$
By construction, the set $\mathcal{X}^x_{\Delta}$ coincides with the set of critical values of a smooth mapping over a space of infinite dimension. In analogy with the classical Sard Theorem in finite dimension (see e.g. \cite{federer69}), the Sard conjecture in sub-Riemannian geometry asserts the following:\\

\noindent {\bf Sard Conjecture.} {\em For every $x\in M$, the set $\mathcal{X}^x_{\Delta}$ has  Lebesgue measure zero in $M$.}\\

The Sard Conjecture cannot be obtained as a straightforward consequence of a general Sard Theorem in infinite dimension, as the latter fails to exist (see for example \cite{BM}). This conjecture remains still largely open, except for  some particular cases in dimension $n\geq 4$ (see \cite{LMOPV,montgomery02,riffordbourbaki})  and for the 3-dimensional case where a stronger conjecture is expected. \\

Whenever $M$ has dimension $3$, the singular horizontal paths can be shown to be contained inside the so-called Martinet surface $\Sigma$ (see Section \ref{sec:framework} below for the definition of the Martinet surface) which happens to be a 2-dimensional set. So, in this case, the Sard Conjecture is trivially satisfied. For this reason, in the 3-dimensional case, the meaningful version of the Sard conjecture becomes the following (here and in the sequel, $\mathcal{H}^s$ denotes the $s$-dimensional Hausdorff measure):\\

\noindent {\bf Sard Conjecture in dimension 3.} {\em For every $x\in M$, $\mathcal{H}^2(\mathcal{X}^x_{\Delta})=0$.}\\

In  \cite{zz95}, Zelenko and Zhitomirskii proved that, for generic rank-two distributions in dimension 3, a stronger version of the Sard conjecture holds.
More precisely, they showed that the Martinet surface is smooth and  the sets $\mathcal{X}^x_{\Delta}$ are locally unions of finitely many smooth curves.
In particular, this implies the generic validity of the Strong Sard Conjecture in dimension 3 (we refer the interested reader to \cite{montgomery02} for a statement of Strong Sard Conjectures in higher dimensions):\\

\noindent {\bf Strong Sard Conjecture in dimension 3.} {\em For every $x\in M$ the set $\mathcal{X}^x_{\Delta}$ has Hausdorff dimension at most $1$.}\\

We note that such a result is the best one can hope for. Indeed, if $y \in \mathcal{X}^x_{\Delta}$ with $y=\gamma(1) \neq x$ for some singular curve $\gamma$, then $\gamma(t) \in \mathcal{X}^x_{\Delta}$ for any $t \in [0,1]$ (this follows by reparameterizing $\gamma$).
Thus, whenever $\mathcal{X}^x_{\Delta}$ contains a point $y \neq x$ then automatically it has at least Hausdorff dimension $1$.

As mentioned above, the results in \cite{zz95} are concerned with generic distributions. Hence, it is natural to investigate what one can say without a genericity assumption, both for the Sard Conjecture and for its Strong version.
Recently, in \cite{BR} the first and fourth authors  proved that the Sard Conjecture in dimension 3 holds whenever:\\
- either $\Sigma$ is smooth;\\
- or some assumption of tangency is satisfied by the distribution over the set of singularities of $\Sigma$.

The aim of the present paper is to show that the description of singular curves given in \cite{zz95} holds in fact for {\em any} analytic distribution in dimension 3.
In particular, we shall prove that the Strong Sard Conjecture holds for any analytic distribution.

To state precisely our result,
we equip $M$ with a Riemannian metric $g$, and for every $x\in M$ and every $L>0$ we denote by $\mathcal{S}^{x,L}_{\Delta,g}$ the set of $\gamma$ in $\mathcal{S}^{x}_{\Delta}$ with length bounded by $L$ (the length being computed with respect to $g$). Then we set
\begin{equation}
\label{eq:XL}
\mathcal{X}^{x,L}_{\Delta,g} := \Bigl\{ \gamma (1) \, \vert \, \gamma \in \mathcal{S}^{x,L}_{\Delta,g}\Bigr\} \subset M.
\end{equation}
We observe that if $g$ is complete, then all the sets $\mathcal{X}^{x,L}_{\Delta,g}$ are compact. Moreover we note that, independently of the metric $g$, there holds
$$
\mathcal{X}^{x}_{\Delta}=\bigcup_{L\in \mathbb N}\mathcal{X}^{x,L}_{\Delta,g}.
$$
Our first result settles the Strong Sard Conjecture in the analytic case. Here and in the sequel, we call {\it singular horizontal curve}  any set of the form $\gamma([0,1])$, where $\gamma:[0,1] \rightarrow M$ is a singular horizontal path. Furthermore, we call {\it semianalytic curve} in $M$ any semianalytic compact connected  subset of $M$ of Hausdorff dimension at most $1$ (see Appendix \ref{App:semianalytic}). 
It is well-known
that semianalytic curves admit a nice stratification into 0-dimensional and 1-dimensional pieces, see Lemma \ref{lem:Stratcurve}.

\begin{theorem}\label{THMSard}
Let $M$ be an analytic manifold of dimension $3$ and $\Delta$ a rank-two totally nonholonomic analytic distribution on $M$. Then any singular horizontal curve is a semianalytic curve in $M$. Moreover, if $g$ is a complete smooth Riemannian metric on $M$ then, for every $x\in M$ and every $L>0$, the set $\mathcal{X}^{x,L}_{\Delta,g}$ is a finite union of singular horizontal curves, so it is a semianalytic curve. In particular, for every $x\in M$, the set $\mathcal{X}^{x}_{\Delta}$ is a countable union of semianalytic curves and it has Hausdorff dimension at most $1$.
\end{theorem}

The proof of Theorem \ref{THMSard} uses techniques from resolution of singularities together with analytic arguments. A crucial step in the proof is to show that the so-called monodromic convergent transverse-singular trajectories (see Definitions \ref{def:conv traj} and \ref{def:SingCharcMonod}) necessarily have infinite length and so cannot correspond to singular horizontal paths. This type of trajectories is a generalization of the singular curves which were investigated by the first and fourth authors at the end of the Introduction of \cite{BR}. If for example the Martinet surface $\Sigma$ is stratified by a singleton $\{x\}$, a stratum $\Gamma$ of dimension 1, and two strata of dimension 2 as in Figure \ref{fig1}, then each $z$ in $\Gamma$ gives rise to such a trajectory $\gamma^z$. We note that, if $\gamma^z$ has finite lenght, then it would correspond to a singular horizontal path from $x$ to $z$. In particular, since the area swept out by the curves $\gamma^z$ (as $z$ varies transversally) is $2$-dimensional, if the curves $\gamma^z$ had finite length then the set $\mathcal{X}^x_{\Delta}$ would have dimension 2 and the example of Figure \ref{fig1} would contradict the Sard Conjecture. As we shall see this is not the case since all the curves $\gamma^z$ have infinite length, so they do not correspond to singular trajectories starting from $x$. We note that, in \cite{BR}, the authors had to understand a similar problem and in that case the lengths of the singular trajectories under consideration were controlled ``by hand'' under the assumption that $\Sigma$ were smooth. Here instead, to handle the general case, we combine resolution of singularities together with a regularity result for transition maps by Speissegger (following previous works by  Ilyashenko).  

\begin{figure}\label{fig1}
\begin{center}
\includegraphics[width=3.5cm]{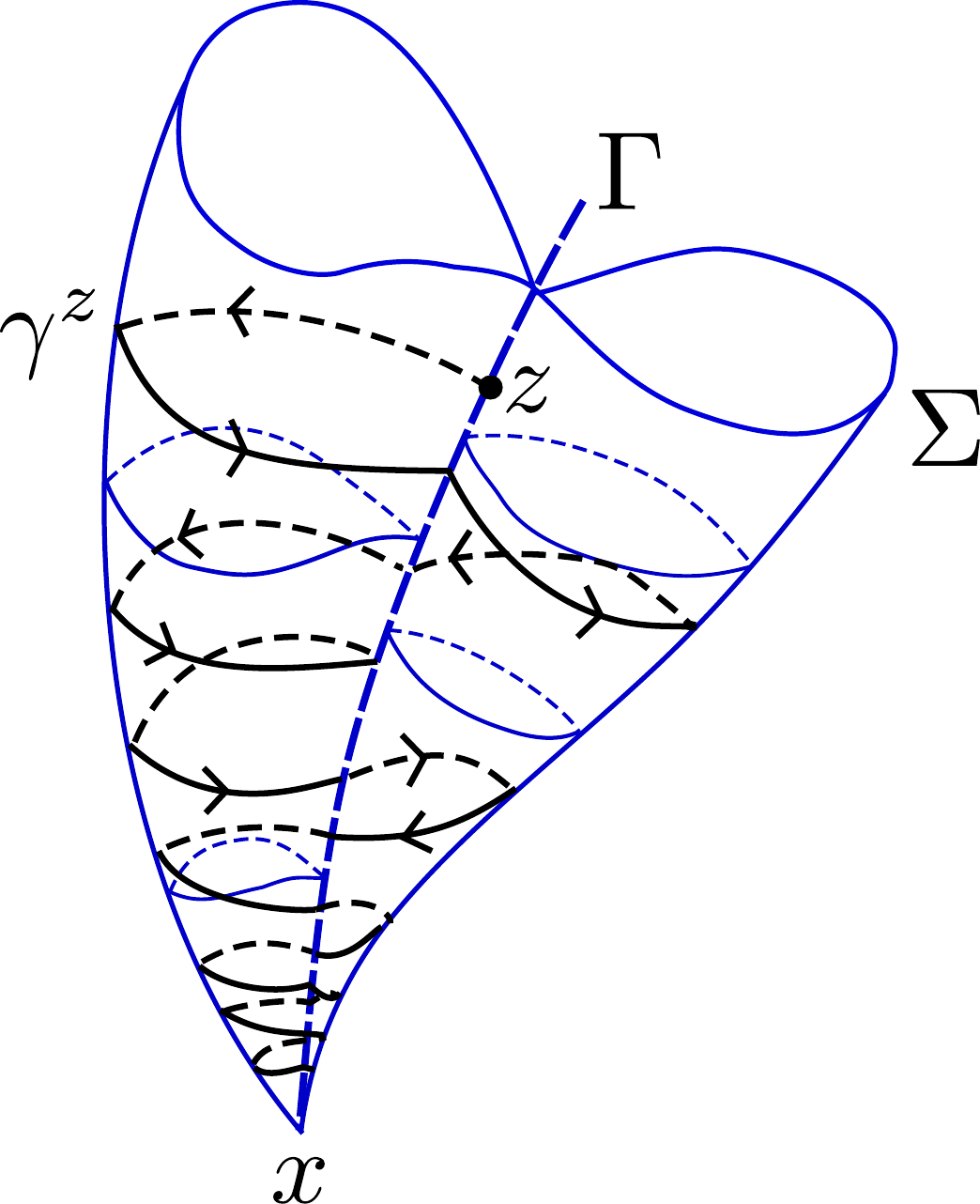}
\caption{Monodromic convergent transverse-singular trajectories}
\end{center}
\end{figure}

Another important step in the proof of Theorem \ref{THMSard} consists in describing the remaining possible singular horizontal paths. We show that the sets $\mathcal{X}^{x,L}_{\Delta,g}$ consist of a finite union of semianalytic curves, which are projections of either characteristic or dicritical orbits of analytic vector fields by an analytic resolution map of the Martinet surface.
A key point in the proof is the fact that the singularities of those vector fields are of saddle type, which holds because of a divergence-type restriction. \\

Theorem \ref{THMSard} allows us to address one of the main open problems in sub-Riemannian geometry, namely the regularity of length-minimizing curves. Given a sub-Riemannian structure $(\Delta,g)$ on $M$, which consists of a totally nonholonomic distribution $\Delta$ and a metric $g$ over $\Delta$, we recall that a {\it minimizing geodesic} from $x$ to $y$ in $M$ is a horizontal path $\gamma:[0,1] \rightarrow M$ which minimizes the energy\footnote{The energy of a horizontal path $\gamma:[0,1] \rightarrow M$ is defined as $\int_0^1 |\dot{\gamma}(t)|^2 \, dt$, where $|\cdot|$ stands for the norm associated to $g$ over $\Delta$.} (and so the length) among all horizontal paths joining $x$ to $y$. It is well-known that minimizing geodesics may be of two types, namely:\\
- either {\em normal},
which means that they are the projection of a trajectory (called normal extremal) of the so-called sub-Riemannian Hamiltonian vector field\footnote{The sub-Riemannian Hamiltonian $H:T^*M \rightarrow \R$ associated with $(\Delta,g)$ in $T^*M$ is defined, in local coordinates, by $$H(x,p):= \max \left\{\frac12
\bigg( \frac{p(v)}{|v|}\bigg)^2 \, \vert \, v \in \Delta_x\setminus \{0\} \right\}\qquad \text{for every $(x,p)$}.
$$
It gives rise to an Hamiltonian vector field, called the sub-Riemannian Hamiltonian vector field, with respect to the canonical symplectic structure on $T^*M$.} in $T^*M$;\\
- or {\em singular}, in which case they are given by the projection of an abnormal extremal (cf. Proposition \ref{PROPsing}).

Note that a geodesic can be both normal and singular. In addition, as shown by Montgomery \cite{montgomery94} in the 1990s, there exist minimizing geodesics which are singular but not normal.  While a normal minimizing geodesic is smooth (being the projection of a trajectory of a smooth dynamical system), a singular minimizing geodesic which is not normal might be nonsmooth. In particular it is widely open whether all singular geodesics (which are always Lipschitz) are of class $C^1$. We refer the reader to \cite{monti14b,riffordbourbaki,Vit} for a general overview on this problem, to \cite{lm08,LLMV13,ty13,monti14a,LLMV15,bcjps18} for some regularity results on singular minimizing geodesics for specific types of sub-Riemannian structures, and to \cite{sussmann15,hl16,mpv16} for partial regularity results for general (possibly analytic) SR structures.

In our setting, the main result of \cite{hl16} can be combined with our previous theorem to obtain the first $C^1$ regularity result for singular minimizing geodesics in arbitrary analytic 3-dimensional sub-Riemannian structures. More precisely, we can prove the following result:

\begin{theorem}\label{THMReg}
Let $M$ be an analytic manifold of dimension $3$, $\Delta$ a rank-two totally nonholonomic analytic distribution on $M$, and $g$ a complete smooth sub-Riemannian metric over $\Delta$. Let $\gamma:[0,1] \rightarrow M$ be a singular minimizing geodesic. Then $\gamma$ is of class $C^1$ on $[0,1]$. Furthermore $\gamma([0,1])$ is semianalytic, and therefore it consists  of finitely many points and finitely many analytic arcs.
\end{theorem}

Theorem \ref{THMReg} follows readily from Theorem \ref{THMSard}, the regularity properties of semianalytic curves recalled in Appendix B, and a breakthrough result of Hakavuori and Le Donne \cite{hl16} on the absence of corner-type singularities of minimizing geodesics. This theorem\footnote{The theorem of Hakavuori and Le Donne \cite[Theorem 1.1]{hl16} is strongly based on a previous result by Leonardi and Monti \cite[Proposition 2.4]{lm08} (see also \cite{LLMV15}) which states that the blow-up of a minimizing geodesic  $\gamma:[0,1] \rightarrow M$ with corner at $t=1/2$ is a broken minimizer made of two half-lines in the tangent Carnot-Carath\'eodory structure at $\gamma(1/2)$. In fact, Proposition 2.4 is not exactly stated in this way in \cite{lm08}. We refer the reader to \cite{mpv17} for a precise statement and a comprehensive and complete proof of \cite[Proposition 2.4]{lm08} as required for the proof of \cite[Theorem 1.1]{hl16}.}  asserts that if $\gamma:[0,1] \rightarrow M$ is a minimizing geodesic which is differentiable from the left and from the right at $t=1/2$ then it is differentiable at $t=1/2$. By Theorem \ref{THMSard} and Lemma \ref{lem:Stratcurve}, if $\gamma:[0,1]\rightarrow M$ is a singular minimizing geodesic, then it is piecewise $C^1$ and so left and right differentiable everywhere\footnote{Except of course at $t=0$ (resp. $t=1$) where $\gamma$ is only right (resp. left) differentiable.} (see Remark \ref{rk:LocalPuiseux} (ii)). Then the main result of \cite{hl16} implies our Theorem \ref{THMReg}.\\

The paper is organized as follows: In Section \ref{Sec:ProofTHMSard}, we introduce some preliminary notions (such as the ones of Martinet surface and characteristic line foliation), and introduce the concepts of characteristic and monodromic transverse-singular trajectories. Section \ref{ProofTHMSard} is devoted to the proof of Theorem \ref{THMSard}, which relies on two fundamental results: first Proposition \ref{prop:TopDichSing}, which provides a clear description of characteristic orbits, and second Proposition  \ref{prop:MonoFiniteLenght}, which asserts that convergent monodromic transverse-singular trajectories have infinite length and so allow us to rule out monodromic horizontal singular paths. The proofs of Proposition \ref{prop:TopDichSing} and of a part of Proposition \ref{prop:MonoFiniteLenght} (namely, Proposition \ref{prop:ComparisonMonodromic}) are postponed to Section \ref{sec:SingCharc}. That  section contains  results on the divergence of vector fields and their singularities, a major theorem on resolution of singularities (Theorem \ref{thm:RSgeneral}), and the proofs mentioned before.  Finally, the four appendices collect some basic results on singular horizontal paths, semianalytic sets, Hardy fields, and resolution of singularities of analytic surfaces and reduction of singularities of planar vector fields.\\

{\bf In the rest of the paper, $M$ is an analytic manifold of dimension $3$, $\Delta$ a rank-two totally nonholonomic analytic distribution on $M$, and $g$ a complete smooth sub-Riemannian metric over $\Delta$.}\\

{\it Acknowledgments:} AF is partially supported by ERC Grant ``Regularity and Stability in Partial Differential Equations (RSPDE)''. AP is partially supported by ANR project LISA (ANR-17-CE40-0023-03). LR is partially supported 
by ANR project SRGI ``Sub-Riemannian Geometry and Interactions'' (ANR-15-CE40-0018).
ABS and AF are thankful for the hospitality of the Laboratoire Dieudonn\'e at the Universit\'e C\^ote d'Azur, where part of this work has been done. We would also like to thank Patrick Speissegger for answering our questions about \cite{speiss2}. 

\section{Characteristic line foliation and singular trajectories}\label{Sec:ProofTHMSard}

\subsection{The Martinet surface}\label{sec:framework}

The {\it Martinet surface} $\Sigma$ associated to $\Delta$ is defined as
$$
\Sigma := \Bigl\{ x\in M \, \vert \, \Delta_x + [\Delta,\Delta]_x\neq T_xM \Bigr\},
$$
where $[\Delta,\Delta]$ is the (possibly singular\footnote{A distribution $\mathcal{D}$ on $M$ is called singular if it does not have constant rank, that is, if the dimension of the vector space $\mathcal{D}_x \subset T_xM$ is not constant.}) distribution given by
$$
[\Delta,\Delta]_x := \Bigl\{ [X,Y](x) \, \vert \, X,Y \mbox{ smooth local sections of } \Delta\Bigr\}.
$$
We recall that the singular curves for $\Delta$ are those horizontal paths which are contained in the Martinet surface $\Sigma$ (see e.g. \cite[Example 1.17 p. 27]{riffordbook}). 

\begin{remark}[Local Model]\label{rk:LocalDescriptMartinet}
Locally, we can always suppose that $M$ coincides with a connected open subset $\mathcal{V}\subset \mathbb{R}^3$, and that $\Delta$ is everywhere generated by global analytic sections. More precisely, we can choose one of the following equivalent formulations:
\begin{itemize}
\item[(i)] $\Delta$ is a totally nonholonomic distribution generated by an analytic 1-form $\delta$ (that is, a section in $ \Omega^1(M)$) and 
\begin{align}\label{eq:martinetequation}
\delta \wedge d\delta =  h \cdot  \omega_M,
\end{align}
where $h$ is an analytic function defined in $M$ whose zero locus defines the Martinet surface (that is, $\Sigma = \{p\in M\,\vert \, h(p)=0\}$) and $\omega_M$ is a local volume form.

\item[(ii)] $\Delta$ is generated by two global analytic vector fields $X^1$ and $X^2$ which satisfy the H\"ormander condition, and $[\Delta,\Delta]$ is generated by $X^1$, $X^2$, and $[X^1,X^2]$. Also, up to using the Flow-box Theorem and taking a linear combination of $X^1$ and $X^2$, we can suppose that
\[
X^1 = \partial_{x_1}, \quad X^2 = \partial_{x_2} + A(x) \, \partial_{x_3}, \quad [X^1,X^2] = A_1(x) \, \partial_{x_3},
\]
where $(x_1,x_2,x_3)$ is a coordinate system on $M$, and $A_1(x) := \partial_{x_1} A(x)$. In this case, the zero locus of $A_1(x)$ defines the Martinet surface (that is, $\Sigma = \{p\in M\,\vert \, A_1(p)=0\}$).
\end{itemize}
\end{remark}

Since $M$ and $\Delta$ are both analytic, the Martinet surface is an analytic set (see e.g. \cite{h73,n66,t71}), and moreover the fact that $\Delta$ is totally nonholonomic implies that $\Sigma$ is a proper subset of $M$ of Hausdorff dimension at most~$2$. Furthermore, we recall that $\Sigma$ admits a global structure of reduced and coherent real-analytic space\footnote{The first author would like to thank Patrick Popescu-Pampu for pointing out that the hypothesis of \cite[Lemma C.1]{BR} is always satisfied in our current framework, that is, when $\Delta$ is non-singular.}, which we denote by $\mathscr{M}$ (see \cite[Lemma C.1]{BR}).

\subsection{Characteristic line foliation}\label{sec:CharacteristicFoliation}

The local models given in Remark \ref{rk:LocalDescriptMartinet} have been explored, for example, in \cite{zz95} and later in \cite[eqs. (2.2) and (3.1)]{BR} in order to construct a locally-defined vector-field whose dynamics characterizes singular horizontal paths at almost every point (cf.  Lemma \ref{lem:CharacteristicFoliation}(ii) below). Since $\Sigma$ admits a global structure of coherent analytic space, these local constructions yield a globally defined (singular\footnote{The foliation $\mathscr{L}$ does not necessarily have rank 1 everywhere, as there may be some points $x\in \Sigma$ where $\mathscr{L}_x=\{0\}$. A point $x$ is called regular if $\mathscr{L}_x$ has dimension 1, and singular if $\mathscr{L}_x=\{0\}$.}) line foliation  $\mathscr{L}$ (in the sense of Baum and Bott \cite[p. 281]{BaumBott}), which we call \emph{characteristic line foliation} (following Zelenko and Zhitomirskii \cite[Section 1.4]{zz95}). More precisely, we have:

\begin{lemma}[Characteristic line foliation]\label{lem:CharacteristicFoliation}
The set 
$$
S := \Bigl\{ p\in \Sigma\,\vert \, p \in Sing(\mathscr{M}) \text{ or }T_p\Sigma \subset \Delta_p\Bigr\}
$$
is analytic of dimension less than or equal to $1$, and there exists a line foliation $\mathscr{L}$ defined over $\Sigma$ such that:
\begin{itemize}
\item[(i)] The line foliation $\mathscr{L}$ is regular everywhere in $\Sigma \setminus S$.
\item[(ii)] If a horizontal path $\gamma :[0,1] \rightarrow M$ is singular with respect to $\Delta$, then its image $\gamma([0,1])$ is contained in $\Sigma$ and it is tangent to $\mathscr{L}$ over $\Sigma \setminus S$, that is 
$$
\gamma(t) \in \Sigma \setminus S\quad \Longrightarrow \quad \dot{\gamma}(t) \in \mathscr{L}_{\gamma(t)} \qquad \mbox{for a.e. } t \in [0,1].
$$
\end{itemize} 
\end{lemma}
\begin{proof}[Proof of Lemma \ref{lem:CharacteristicFoliation}]
$S$ is analytic because $\Delta$ and $\Sigma$ are both analytic. The total nonholonomicity of $\Delta$ implies that $S$ has dimension smaller than or equal to $1$, 
see Lemma 2 of \cite{BR}.

Let $i: \mathscr{M} \to M$ be the inclusion. Since $\Delta$ is a coherent sub-sheaf of $\Omega^1(M)$ (cf. Remark \ref{rk:LocalDescriptMartinet}), the pull-back $\mathscr{L} := i^{\ast}(\Delta)$ is also a coherent sub-sheaf of $\Omega^1(\mathscr{M})$. Furthermore, since $\Delta$ is everywhere locally generated by one section, so is $\mathscr{L}$. It is thus enough to study $\mathscr{L}$ locally. 

Fix a point $p \in \Sigma$.  If $\Sigma$ has dimension smaller than or equal to $1$ at $p$,
then $\Sigma = S$ in a neighborhood of $p$ and the claim of lemma holds trivially.  
If $\Sigma$ has dimension $2$ at $p$, then $\mathscr{L}$ generates a line foliation over 
a neighborhood of $p$ in $\Sigma$.

To prove (i)  fix a point $p$ where $\mathscr{M}$ is smooth (in particular, $\Sigma$ is smooth as a subset of $M$) and $\Delta_p + T_p\Sigma = T_pM$. Then there exists a local coordinate system $(x_1,x_2,x_3)$ centered at $p$ so that $\Sigma = \{x_3=0\}$ and $\delta = dx_1 + A(x) dx_2$, therefore $\mathscr{L}$ is regular at $p$.

Finally, assertion (ii) follows from the above formulae in local coordinates and the characterization of singular horizontal paths given in Proposition \ref{PROPcharacteristic}.
\end{proof}

\begin{remark}[Characteristic vector-field]\label{rk:LocalModelFoliation}
We follow \cite[eq. (3.1)]{BR}. In the notation of Remark \ref{rk:LocalDescriptMartinet}(ii), let $h$ be a reduced analytic function whose zero set is equal to the Martinet surface $\Sigma$. 
Consider the vector-field
\[
\mathcal{Z} := X^1(h) X^2 - X^2(h)X^1.
\]
Then the restriction of $\mathcal{Z}$ over $\Sigma$ is a generator of the line foliation $\mathscr{L}$.
\end{remark}

\subsection{Stratification of the Martinet surface}\label{sec:StratMartinet}

By a result of {\L}ojasiewicz \cite{lojasiewicz},
every analytic set $X$ (or an analytic space) admits a semianalytic stratification into non-singular strata.  Each stratum of such stratification is a locally closed analytic submanifold of $X$ and a semianalytic subset of $X$.  Furthermore, it is always possible to choose such stratification  {\em regular}, i.e., that satisfies Whitney regularity conditions, cf.  \cite{lojasiewicz} or \cite{wall}.   For our purpose we need a stratification of the Martinet surface $\Sigma$ 
that, in addition,  is compatible with the distribution $\Delta$ in the sense of the following lemma.  


\begin{lemma}[Stratification of $\Sigma$]\label{lem:StratFol}
There exists a regular semianalytic stratification of $\Sigma$,
$$
\Sigma = \Sigma^0 \cup  \Sigma^1_{tr} \cup \Sigma^1_{tan} \cup  \Sigma^2,
$$
which satisfies the following properties:
\begin{itemize}
\item[(i)] $S = \Sigma^0 \cup  \Sigma^1_{tr} \cup \Sigma^1_{tan}$ (cf. Lemma \ref{lem:CharacteristicFoliation}).
\item[(ii)] $\Sigma^0$ is a locally finite union of points.
\item[(iii)]  $\Sigma^1_{tan}$ is a locally finite union of $1$-dimensional strata with tangent spaces everywhere contained in $\Delta$ (that is, $T_p\Sigma^1_{tan} \subset \Delta_p$ for all $p\in \Sigma^1_{tan}$).
\item[(iv)]  $\Sigma^1_{tr}$ is a locally finite union of $1$-dimensional strata transverse to $\Delta$ (that is, $T_p\Sigma^1_{tr} \oplus \Delta_p = T_pM$ for all $p\in \Sigma^1_{tr}$);
\item[(v)]  $\Sigma^2$ is a locally finite union of $2$-dimensional strata transverse to $\Delta$ (that is, $T_p\Sigma^2 + \Delta_p = T_pM$ for all $p\in \Sigma^2$). 
\end{itemize}
Moreover,   every 1-dimensional stratum $\Gamma$ satisfies the following local triviality property:  For each point $p$ in $\Gamma$ there exists a neighborhood $\mathcal{V} $ of $p$ in $M$ such that  $\Sigma^2 \cap \mathcal V$ is the disjoint union of finitely many 2-dimensional analytic submanifolds  $\Pi_1, \ldots, \Pi_r$ ($\Sigma^2 \cap \mathcal V$ could be empty) such that for each $i$, 
$\Pi_i \cup \Gamma$ is a closed $C^1$-submanifold of $\mathcal V$ with boundary, denoted by $\overline \Pi_i$, with $\Gamma = \partial \overline \Pi_i$. 
\end{lemma}

\begin{proof}[Proof of Lemma \ref{lem:StratFol}]
By Lemma \ref{lem:CharacteristicFoliation}, the set $\Sigma^2:= \Sigma \setminus S$ is smooth and $\mathscr{L}$ is non-singular everywhere over it. Now, we recall that $S$ is an analytic set of dimension at most $1$, so it admits a semianalytic stratification $S_0 \cup S_1$, where $S_0$ is a locally finite union of points and $S_1$ is a locally finite union of (open) analytic curves. Moreover, by \cite{lojasiewicz} or \cite{wall}, we may assume that $\Sigma ^2$, $S_1$, and $S_0$ is a regular stratification of 
$\Sigma$.  

Fixed a $1$-dimensional stratum $\Gamma$ in $S_1$, its closure $\overline{\Gamma}$ is a closed semianalytic set. The condition $T_p\Gamma \subset \Delta_p$ 
is semianalytic (given, locally, in terms of analytic equations and inequalities). Therefore, up to removing from $\Gamma$ a locally finite number of points, 
we can assume that:\\
- either $\Delta_p$ contains $T_p{\Gamma}$ for every $p\in \Gamma$;\\
- or $\Delta_p$ transverse 
to  $T_p\Gamma$ for every $p\in \Gamma$.\\
In other words, up to adding a locally finite union of points to $S_0$, we can suppose that the above dichotomy is constant along connected components of $S_1$. Then, it suffices to denote by $\Sigma^1_{tan}$ the subset of $S_1$ consisting of all connected components where $T_p\Gamma \subset \Delta_p$ for every point $p\in \Gamma$, and by $\Sigma^1_{tr}$ the subset of all connected components where the transversality condition  $\Delta_p \oplus T_p\Gamma = T_p M$ holds.  

The last claim of Lemma follows from \cite{wall}.
\end{proof}

\begin{remark}[Puiseux with parameter]\label{rk:LocalTrivialityPuiseux}
As follows from \cite[Proposition 2]{pawlucki} or  \cite{wall} (proof of Proposition p.342), we may require in Lemma \ref{lem:StratFol} the following stronger version of local triviality of $\Sigma$ along $\Gamma$.  Given 
$p \in \Gamma$,  there exist a positive integer $k$ and  a local system of analytic coordinates 
$x=(x_1,x_2,x_3)$ at $p$ such that $\Gamma = \{x_2=x_3=0\}$ and each $\Pi_i$ is the graph $x_3= \varphi_i (x_1,x_2)$, defined locally on $\{(x_1,x_2)\, \vert \, x_2\ge 0 \}$ (or 
$\{(x_1,x_2)\, \vert \, x_2\le 0 \}$), such that $\varphi_i $ is $C^1$ and the mapping $(t,x_1)\mapsto \varphi_i  (x_1, t^k)$ is analytic.\\
One may remark that the latter two conditions imply that, in fact, $\varphi_i$ is of class $C^{1,1/k}$.  Indeed,  we may write  for $x_2\ge 0$
$$
\varphi_i (x_1,x_2) = \sum _{i \in \N,j\in \N}  a_{i,j} x_1^i x_2 ^{j/k}.
$$
The fact that $\varphi_i $ is $C^1$ implies that in this sum $j=0$ or $j\ge k$.  
Therefore the derivative $\partial \varphi_i/ \partial x_2$ is H\"older continuous with exponent 
$1/k$.
\end{remark}

By the local triviality property stated in Lemma \ref{lem:StratFol} and by Remark \ref{rk:LocalTrivialityPuiseux}, the restriction of $\Delta$ to a neighborhood of a point of $\Sigma^1_{tr}$ satisfies the following property (we recall that $M$ is equipped with a metric $g$):

\begin{figure}\label{fig2}
\begin{center}
\includegraphics[width=4.2cm]{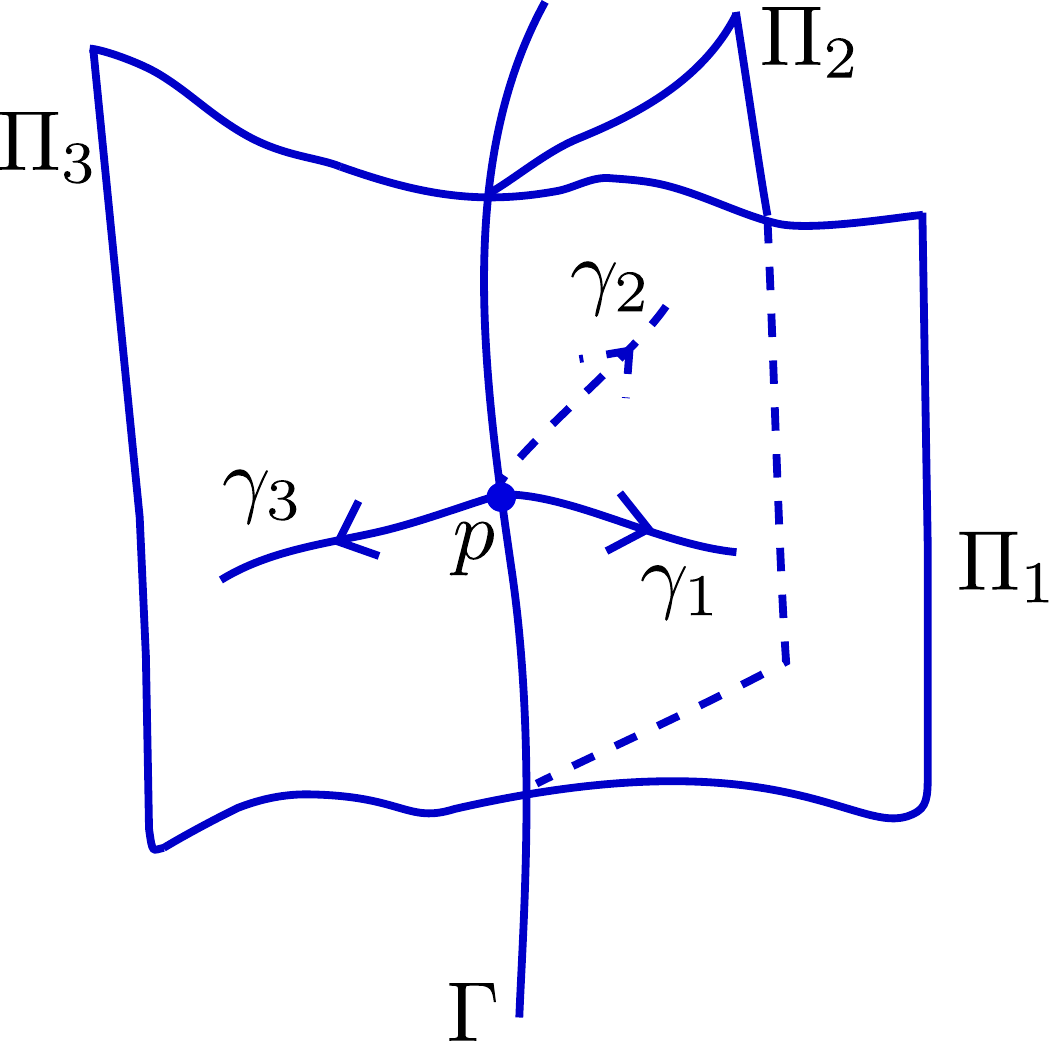}
\caption{$\mathcal{X}^{p,\epsilon}_{\Delta,g}$ for $p\in \Sigma_{tr}^1$ and $\epsilon>0$ small}
\end{center}
\end{figure}

\begin{lemma}[Local triviality of $\Delta$ along $\Sigma_{tr}^1$]\label{rk:LocalTriviality}
Let $\Gamma$ be a $1$-dimensional stratum in $\Sigma^1_{tr}$ and  let $p \in \Gamma$ be fixed.  Then the following properties hold:
\begin{itemize}
	\item[(i)] There exists a neighborhood $\mathcal{V}$ of $p$ and $\delta>0$ such that, for every point $q\in \mathcal{V} \cap \Sigma_{tr}^1$ and every injective singular horizontal path $\gamma:[0,1]\to \Sigma$ such that $\gamma(0)=q$, $\gamma(1) \in \Sigma_{tr}^1$, and $\gamma((0,1)) \subset \Sigma^2$, the length of $\gamma$ is larger than $\delta$.
\item[(ii)] The image of a singular horizontal path $\gamma: [0,1] \to M$ such that $\gamma([0,1)) \subset \Sigma^{2}$ and $\gamma(1)\in \Sigma_{tr}^1$  is semianalytic.
\end{itemize}
In particular, if $\mathcal{V} $ is a neighborhood of $p$ in $M$ such that  $\Sigma^2 \cap \mathcal V$ is the disjoint union of the 2-dimensional analytic submanifolds  $\Pi_1, \ldots, \Pi_r$ as in Lemma \ref{lem:StratFol}, then  for $\epsilon>0$ small enough there are singular horizontal paths $\gamma_1, \ldots, \gamma_r :[0,1] \rightarrow \Sigma$ with $\gamma_i(0)=p$ and $\gamma_i((0,1]) \subset \Pi_i$ for $i=1, \ldots,r$ such that (see Figure 2)
$$
\mathcal{X}^{p,\epsilon}_{\Delta,g} = \bigcup_{i=1}^r \gamma_i\left( [0,1]\right),
$$
where $\mathcal{X}^{p,\epsilon}_{\Delta,g}$ is defined in \eqref{eq:XL}.
\end{lemma}

\begin{proof}[Proof of Lemma \ref{rk:LocalTriviality}]
The lemma follows readily from the following observation.  
Let $x=(x_1,x_2,x_3)$ denote the system of coordinates at $p$ introduced in  
Remark \ref{rk:LocalTrivialityPuiseux}.  Suppose that the distribution $\Delta$ is locally defined by the $1$-form 
$\delta$ as in Remark \ref{rk:LocalDescriptMartinet}.  Then the  pull-back of $\delta$ on $\Pi_i$ by the map $(x_1,t) \to (x_1,t^k, \varphi_i (x_1, t))$ is an analytic $1$-form: 
$\delta_i= a(x_1,t) dx_1 + b(x_1,t) dt$.  The condition of transversality of $\Delta$ and $\Gamma$ at $p$ means $a(0,0) \ne 0$ and therefore the integral curves of $\Delta_i$, 
that is the singular horizontal path of $\Delta$, are uniformly transverse to $\Gamma$ 
in a neighborhood of $p$.
\end{proof}

It remains now to introduce some definitions related to singular horizontal paths or more precisely singular trajectories (i.e., trajectories of the characteristic line foliation) converging to the set 
\begin{equation}
\label{eq:tilde Sigma}
\widetilde{\Sigma} := \Sigma^0 \cup \Sigma^1_{tan}.
\end{equation} 
This is the purpose of the next section.

\subsection{Characteristic and monodromic transverse-singular trajectories}\label{SECts}

We restrict our attention to a special type of trajectories of the characteristic foliation $\mathscr{L}$. 

\begin{definition}[Convergent transverse-singular trajectory]
\label{def:conv traj}
We call \emph{transverse-singular trajectory} any absolutely continuous path $\gamma:[0,1) \rightarrow \Sigma$  such that
$$
\dot{\gamma}(t) \in \mathscr{L}_{\gamma(t)} \qquad \mbox{for a.e. } t \in [0,1),
$$
and 
$$
\gamma(t) \in \Sigma^2 \cup \Sigma^1_{tr} \qquad \forall\,t \in [0,1).
$$
Moreover, we say that $\gamma$ is  \emph{convergent} if it admits a limit as $t$ tends to $1$. 
\end{definition}

We are going to introduce a dichotomy between two types of convergent transverse-singular trajectories which is inspired by the following well-known result (see \cite[Theorem 9.13]{IY} and \cite[Definitions 9.4 and 9.6]{IY}):

\begin{proposition}[Topological dichotomy for planar analytic vector-fields]\label{prop:TopDichPlanar}
Let $\mathcal{Z}$ be an analytic vector field defined over an open neighborhood $U$ of 
the origin $0$ in $\mathbb{R}^2$, and suppose that $0$ is a singular point 
of $\mathcal{Z}$. Given a regular orbit $\gamma(t)$ of $\mathcal{Z}$ converging to $0$, then:
\begin{itemize}
\item[(i)] either $\gamma$ is a \emph{characteristic orbit}, that is, the secant curve $\psi(t) := \gamma(t) /|\gamma(t)| \in \mathbb{S}^{1}$ has a unique limit point;
\item[(ii)] or $\gamma$ is a \emph{monodromic orbit}, that is, there exists an analytic section $\Lambda$ of the vector-field $\mathcal{Z}$ at $0$\footnote{In other words, $\Lambda$ is a connected segment whose boundary contains $0$ and the vector field $\mathcal{Z}$ is transverse to $\Lambda$ everywhere out of $0$.} such that $\gamma \cap \Lambda$ is the disjoint union of an infinite number of points. 
\end{itemize}
\end{proposition}



Here is our definition. 

\begin{definition}[Characteristic and monodromic convergent transverse-singular trajectories]\label{def:SingCharcMonod}
Let  $\gamma: [0,1) \rightarrow \Sigma$ be a convergent transverse-singular trajectory such that $\bar{y}:=\lim_{t\rightarrow 1} \gamma(t)$ belongs to $\widetilde{\Sigma}$ (see \eqref{eq:tilde Sigma}). Then we say that:
\begin{itemize}
\item[(i)] 
$\gamma$ is \emph{monodromic} if there exists a section $\Lambda \subset \Sigma$ of $\mathscr{L}$ at $\bar{y}$\footnote{That is, $\Lambda$ is a connected 1-dimensional semianalytic manifold with boundary contained in $\Sigma$, whose boundary contains $\bar{y}$ and such that $\Lambda \setminus\{\bar{y}\} \subset \Sigma^2\cup \Sigma^1_{tr}$ is everywhere transverse to $\mathscr{L}$.} such that $\gamma([0,1)) \cap \Lambda$ is the disjoint union of infinitely many points. In addition, we say that $\gamma$ is \emph{final} if $\gamma([0,1)) \cap \Sigma^1_{tr}$ is empty or infinite. In the latter case, we may choose as $\Lambda$ a branch of $\Sigma^1_{tr}$.
\item[(ii)] 
$\gamma$ is \emph{characteristic} if it is not monodromic.
\end{itemize}
\end{definition}

From now on, we call monodromic  (resp. characteristic) trajectory any convergent transverse-singular trajectory with a limit in $\widetilde{\Sigma}$ which is monodromic (resp. characteristic). The next section is devoted to the study of characteristic and monodromic trajectories, and to the proof of Theorem \ref{THMSard}.



\section{Proof of Theorem \ref{THMSard}}\label{ProofTHMSard}

The proof of Theorem \ref{THMSard} proceeds in three steps. Firstly, we describe some properties of regularity and finiteness satisfied by the characteristic trajectories. Secondly, we rule out monodromic trajectories as possible horizontal paths starting from the limit point. Finally, combining all together, we are able to describe precisely the singular horizontal curves and the sets of the form  $\mathcal{X}^{x,L}_{\Delta,g}$ (see \eqref{eq:XL}). 

\subsection{Description of characteristic trajectories}

The following result is  a consequence of the results on resolution of singularities stated in Theorem \ref{thm:RSgeneral} and the fact that the characteristic trajectories correspond, in the resolution space, to characteristics of an analytic vector field with singularities of saddle type. 

\begin{proposition}\label{prop:TopDichSing}
Let $\Sigma^0$ and $\widetilde \Sigma$ be as in Lemma \ref{lem:StratFol} and \eqref{eq:tilde Sigma}.
There exist a locally finite set of points $\widetilde{\Sigma}^0$, with $\Sigma^0 \subset \widetilde{\Sigma}^0\subset \widetilde{\Sigma}$, such that the following properties hold: 
\begin{itemize}
\item[(i)] If  $\gamma: [0,1) \rightarrow \Sigma$ is a convergent transverse-singular trajectory such that $\bar{y}:=\lim_{t\rightarrow 1} \gamma(t)$ belongs to $\widetilde{\Sigma}$ 
then $\bar{y}$ belongs to $ \widetilde{\Sigma}^0$.  Moreover, if $\gamma$ is characteristic then $\gamma ([0,1))$ is 
semianalytic  and there is  $\bar{t} \in [0,1)$ such that $\gamma([\bar{t},1)) \subset \Sigma^2$.
\item[(ii)] For every $\bar{y} \in \widetilde{\Sigma}^{0}$ there exists only finitely many (possibly zero) characteristic  trajectories converging to $\bar{y}$ and all of them are semianalytic curves.
\end{itemize}
\end{proposition}

The proof of Proposition \ref{prop:TopDichSing} is given in subsection \ref{sec:ThmTopDichSing}, as a consequence of Theorem \ref{thm:RSgeneral}.

\begin{remark}[On Proposition \ref{prop:TopDichSing} and its proof]\label{rk:ProofThmTopDichSing}
\hfill
\begin{itemize}

\item[(i)] There exist elementary proofs of Proposition \ref{prop:TopDichPlanar}. However, for proving Proposition \ref{prop:TopDichSing}(ii) we cannot use the topological simplicity of the plane. In this context it is natural to use resolution of singularities (c.f. \cite[Proof of Theorem 9.13]{IY}).


\item[(ii)] Proposition \ref{prop:TopDichSing}(ii) is specific to characteristic line foliations, and does not hold for arbitrary line foliations over a surface. In our situation we can show that there exists a (locally defined) vector field which generates the characteristic foliation $\mathscr{L}$ and whose divergence is controlled by its coefficients (see subsection \ref{sec:divergence}, cf. \cite[Lemmas 2.3 and 3.2]{BR}). This guarantees that, after resolution of singularities, all singular points of the pull back of $\mathscr{L}$ are saddles (see Theorem \ref{thm:RSgeneral}(II), cf. Lemma \ref{lem:final}). 
\end{itemize}
\end{remark}

\subsection{Monodromic trajectories have infinite length}\label{sec:Monodromic}

The main objective of this subsection is to prove the following crucial result:

\begin{proposition}[Length of monodromic trajectories]\label{prop:MonoFiniteLenght}
The length of any monodromic trajectory is infinite.  
\end{proposition}


\begin{remark}\label{rk:ProofThmMonoFiniteLenght}
If we assume that  the distribution $\Delta$ is generic (with respect to the $\mathcal{C}^{\infty}$-Whitney topology), then the Martinet surface is smooth and the above result corresponds to \cite[Lemma 2.1]{zz95}.
\end{remark}

The proof of Proposition \ref{prop:MonoFiniteLenght} is done by contradiction. The first step consists in showing that if $\gamma$ has finite length, then every monodromic  trajectory which is ``topologically equivalent" to $\gamma$ (see Definition \ref{def:Jumps} below) also has finite length (see Proposition \ref{prop:ComparisonMonodromic} below). Hence, as discussed in the introduction, the assumption of finiteness on the length of $\gamma$ implies that $\mathcal{X}^{\bar{y}}_{\Delta}$ has positive $2$-dimensional Hausdorff measure (cf. Lemma \ref{lem:JumpsInterval}). Then, the second step consists in using an analytic argument based on Stokes' Theorem to obtain a contradiction.\\

Let us consider a monodromic trajectory  $\gamma:[0,1) \rightarrow \Sigma$ with limit $\bar{y} \in \widetilde{\Sigma}$ and assume that  $\gamma$ is injective and  final (cf. Definition \ref{def:SingCharcMonod}(i)), and that its image is contained in a neighborhood $\mathcal{V}$ of $\bar{y}$ where the line foliation $\mathscr{L}$ is generated by a vector-field $\mathcal{Z}$ (see Remark \ref{rk:LocalModelFoliation}). Denote by $\varphi^{\mathcal{Z}}_s(x)$ the flow associated to $\mathcal{Z}$ with time $s$ and initial condition $x\in \mathcal{V} \cap \Sigma$, by $\Lambda$ a fixed section as in Definition \ref{def:SingCharcMonod}, and by $d^{\Lambda}: \Lambda \to \mathbb{R}$  the function which associates to each point $p \in \Lambda$ the length of the half-arc contained in $\Lambda$ which joins $p$ to $\bar{y}$ (we may also assume that $\Lambda \cap \mathcal{V}$ is a curve connecting $\bar{y}$ to a point of the boundary of $\mathcal{V}$). Moreover assume that $\gamma(0)$ belongs to $\Lambda$. By monodromy, there exists an infinite increasing sequence $\{t_k^{\gamma}\}_{k\in \N}$ in $[0,1)$ with $t_0^{\gamma}=0$ such that 
$$
 \gamma(t) \in \Lambda \mbox{ if and only if } t=t_k^{\gamma} \mbox{ for some } k \in \N
 $$
 and
$$
 \lim_{k \rightarrow \infty} t_k^{\gamma}= 1.
$$
 We are going now to introduce a sequence of Poincar\'e mappings adapted to $\gamma$, we need to distinguish two cases, depending whether the set $\gamma ([0,1)) \cap \Sigma_{tr}^1$ is finite or not.
 Note that, if $\gamma ([0,1)) \cap \Sigma_{tr}^1$ is a finite set, then up to restricting $\gamma$ to an interval of the form $[t_0,1)$ for some $t_0\in [0,1)$,
 we can assume that  $\gamma ([0,1)) \cap \Sigma_{tr}^1=\emptyset$. Hence the two cases to analyze are the case where $\gamma ([0,1)) \cap \Sigma_{tr}^1$ is empty and the case where $\gamma ([0,1)) \cap \Sigma_{tr}^1$ is infinite.\\


\begin{figure}\label{fig3}
\begin{center}
\includegraphics[width=4.2cm]{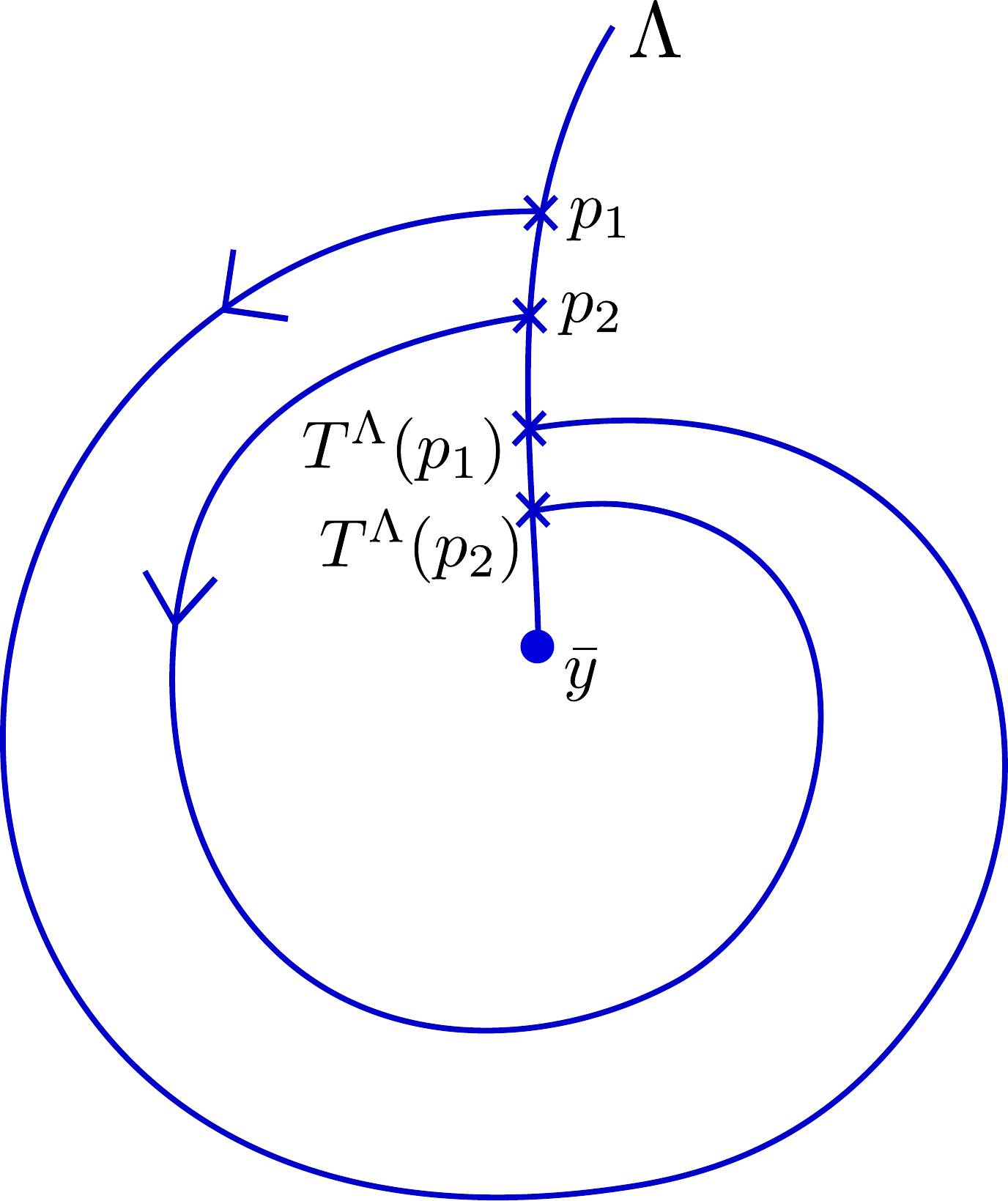}
\caption{The case $\gamma ([0,1)) \cap \Sigma_{tr}^1=\emptyset$}
\end{center}
\end{figure}

\noindent {\bf First case:} $\gamma ([0,1)) \cap \Sigma^1_{tr} = \emptyset$. \\
This is the classical case where we can consider the Poincar\'{e} first return map from $\Lambda$ to $\Lambda$ (see e.g. \cite[Definition 9.8]{IY}). By a Poincar\'e-Bendixon type argument, up to shrinking $\mathcal{V}$ and changing the orientation of $\mathcal{Z}$ we may assume that the mapping 
$$
T^{\Lambda} \, : \, \Lambda \cap \mathcal{V} \, \longrightarrow \, \Lambda \cap \mathcal{V}
$$
which assigns to each $p\in \Lambda \cap \mathcal{V}$ the first point $\varphi^{\mathcal{Z}}_{t}(p) \in \Lambda$ with $t>0$ and $d^{\Lambda}(\varphi^{\mathcal{Z}}_{t}(p))<d^{\Lambda}(p)$ is well-defined, continuous,  and satisfies 
\begin{eqnarray}\label{18sept1}
T^{\Lambda}(\gamma(t_k^{\gamma})) = \gamma(t_{k+1}^{\gamma}) \qquad \forall\,k \in \N
\end{eqnarray}
 and
\begin{equation}\label{eq:PoincareTransitionProp}
d^{\Lambda}(p_1) < d^{\Lambda}(q) < d^{\Lambda}(p_2)  \quad \implies \quad  d^{\Lambda}(T^{\Lambda}(p_1)) < d^{\Lambda}(T^{\Lambda}(q)) < d^{\Lambda}(T^{\Lambda}(p_2))
\end{equation}
for every $p_1, p_2, q$ in $\Lambda \cap \mathcal{V}$.\\



\begin{figure}\label{fig4}
\begin{center}
\includegraphics[width=5.3cm]{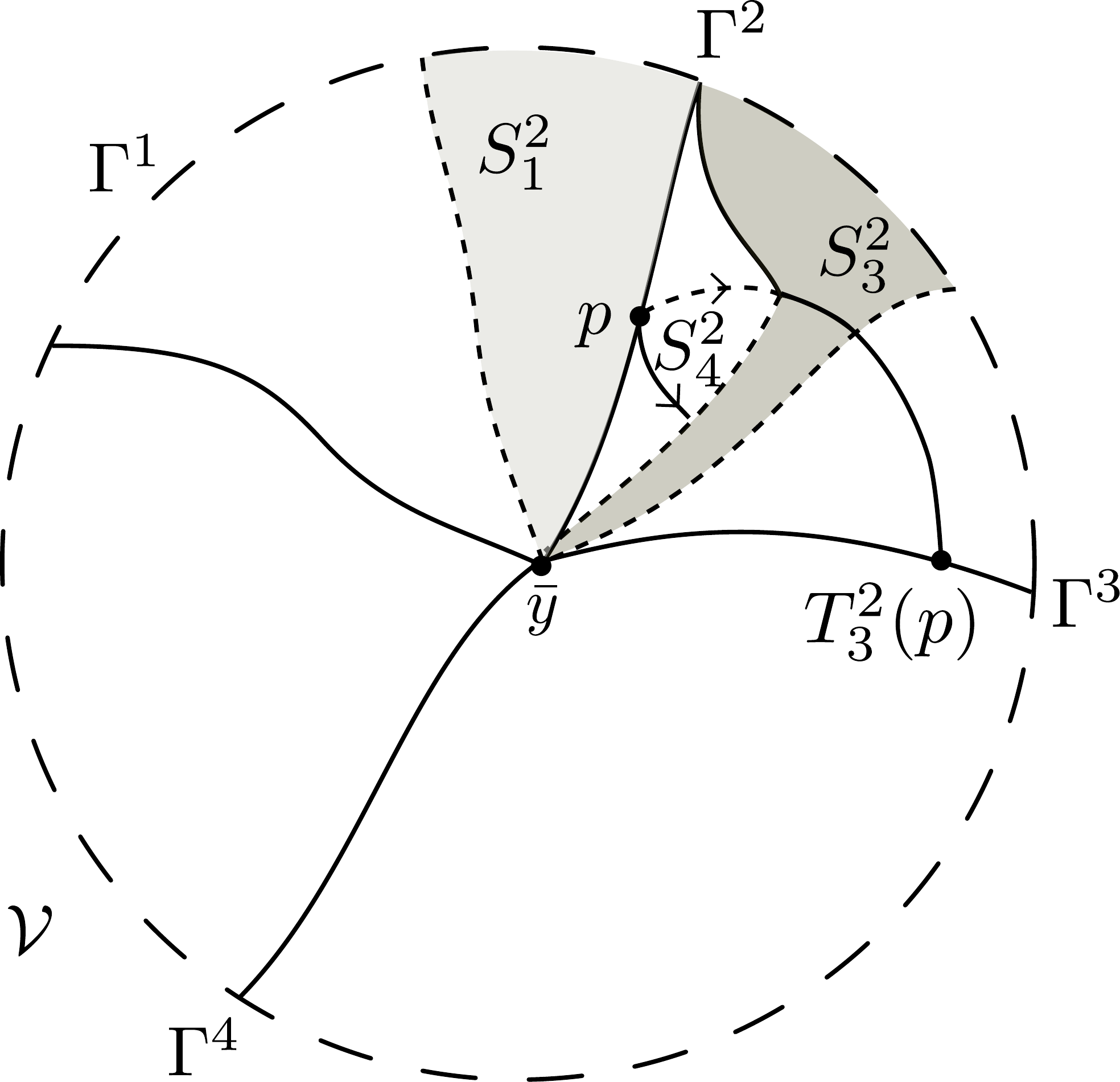}
\caption{The case $\gamma ([0,1))\cap \Sigma_{tr}^1$ infinite}
\end{center}
\end{figure}

\noindent {\bf Second case:} $\gamma ([0,1)) \cap \Sigma^1_{tr}$ is infinite. \\
In this case, up to shrinking $\mathcal{V}$, by semianalyticity of $\Sigma^{1}_{tr}$ and Lemma \ref{lem:StratFol} we can assume that $\Sigma^1_{tr}\cap \mathcal{V}$ is the union of $r$ connected components, say $\Gamma^1, \ldots, \Gamma^r$, whose boundaries are given by $\bar{y}$ and a point in the boundary of $\mathcal{V}$ (this point is distinct for each $i=1,\ldots,r$). In addition,  for each $i=1, \ldots, r$ there exists a neighborhood $\mathcal{V}^i$ of $\Gamma^i$ such that $\left(\Sigma \setminus \Gamma^i\right) \cap \mathcal{V}^i$ is the union of $s_i$ connected smooth subsets of $\Sigma^2$, say $S_{j}^i$ for $j = 1, \ldots, s_i$. Furthermore, as in the first case and up to shrinking $\mathcal{V}$ again,  by a Poincar\'e-Bendixon type argument we may assume that for every $i=1, \ldots, r$, if a piece of $\gamma([0,1))$ joins $\Gamma^i$ to some $\Gamma^{i'}$ through some $S^i_j$ then the corresponding Poincar\'e mapping from $\Gamma^{i}$ to $\Gamma^{i'}$ is well-defined. 
To be more precise,  for  each $i=1, \ldots,r$ we consider the maximal subset of the $S^i_j$'s,  relabeled $S^i_1, \ldots, S^i_{\hat{s}_i}$, with the jump correspondence   
$$
J(i,\cdot) \, : \, j \in \left\{1, \ldots, \hat{s}_i\right\} \, \longrightarrow \, \{1, \ldots, r\},
$$
such that the transition maps 
$$
T^{i}_j \, : \, \Gamma^i \longrightarrow  \Gamma^{J(i,j)} \qquad \forall\,i=1, \ldots, r, \quad \forall\,j=1, \ldots, \hat{s}_i,
$$
that  assign to each $p\in \Gamma^i$ the point $q\in \Gamma^{J(i,j)}$ such that there is an absolutely continuous path $\alpha:[0,1] \rightarrow \Sigma$ tangent to $\mathscr{L}$ over $(0,1)$ satisfying $\alpha(0)=p$, $\alpha(1)=q$, $\alpha((0,1)) \subset \Sigma^2$ and $\alpha((0,\epsilon)) \subset S^i_j$ for some $\epsilon >0$, are well-defined and continuous. Similarly as before, if we denote by  $d^i:\Gamma^i \to \mathbb{R}$ the function which associates to each point $p$ the length of the half-arc contained in $\Gamma^i$ which joins $p$ to $\bar{y}$, then we may also assume that for every $i=1, \ldots,r $ and every $p,q \in \Gamma^i$,
\begin{equation}\label{eq:TransitionPropI}
d^{i}(p) < d^{i}(q) \quad \implies \quad  d^{J(i,j)}(T^{i}_j(p)) < d^{J(i,j)}(T^{i}_j(q)) \qquad \forall\,j=1, \ldots, \hat{s}_i.
\end{equation}
By construction, for each integer $k$, there are $i_k \in \{1, \ldots,r\}$ and $j_k \in \{1, \ldots, \hat{s}_{i_k}\}$ such that $\gamma (t_k^{\gamma}) \in \Gamma^{i_k}$ and $\gamma (t_{k+1}^{\gamma}) \in \Gamma^{J(i_k,j_k)}=\Gamma^{i_{k+1}}$. We call sequence of jumps of $\gamma$ the sequence $\{(i_k,j_k)\}_{k\in \N}$  associated with $\{t_k^{\gamma}\}_{k\in \N}$. \\

We can now introduce the equivalence class on the set of monodromic trajectories.

\begin{definition}[Equivalence of monodromic paths]\label{def:Jumps}
Let $\gamma_1, \gamma_2: [0,1) \rightarrow \Sigma \cap \mathcal{V}$ be two final and injective monodromic trajectories with the same limit point $\bar{y}$ and which share the same section $\Lambda$, where $\gamma_i(0) \in \Lambda$ for $i=1,2$. We say that $\gamma_1$ and $\gamma_2$ are jump-equivalent if:\\
- either $\gamma_1([0,1)) \cap 
\Sigma^{1}_{tr} = \gamma_2 ([0,1)) \cap \Sigma_{tr}^1 = \emptyset$;\\
- or they have the same sequence of jumps.
\end{definition}

By classical  considerations about the Poincar\'{e} map $T^{\Lambda}$ defined in the first case or by a concatenation of orbits of $\mathscr{L}$ connecting $\Gamma^i$ to $\Gamma^j$ in the second case, the following holds:

\begin{lemma}[One parameter  families of equivalent monodromic paths]\label{lem:JumpsInterval}
Let $\gamma : [0,1) \rightarrow \Sigma \cap \mathcal{V}$ be a final and injective monodromic trajectory  with limit point $\bar{y}$, and let $\Lambda$ be a section such that $\gamma(0) \in \Lambda$. Then, for every point $p \in \Lambda$ with $d^{\Lambda}(p)<d^{\Lambda}(\gamma(0))$, there exists a final and injective monodromic trajectory $\lambda:[0,1) \to \Sigma \cap \mathcal{V}$, with $\lambda(0)=p$, which is jump-equivalent to $\gamma$. Moreover, such a trajectory is unique as a curve (that is, up to reparametrization).

%
%
\end{lemma}

Lemma \ref{lem:JumpsInterval} plays a key role in the proof of Proposition \ref{prop:MonoFiniteLenght}. Indeed,
from the existence of one monodromic trajectory, it allows us to infer the existence of a parametrized set of monodromic trajectories filling a $2$-dimensional surface. The next result will also be crucial to control the length of the monodromic trajectories in such a set (we denote by $\mbox{length}^g$ the length of a curve with respect to the metric $g$).

\begin{proposition}[Comparison of equivalent monodromic paths]\label{prop:ComparisonMonodromic}
Let $\gamma$ be a monodromic trajectory with limit point $\bar{y}$ and section $\Lambda$ such that $\gamma(0) \in \Lambda$. Suppose that the length of $\gamma$ is finite. Then there exists a constant $K>0$ such that, for every monodromic trajectory  $\lambda$ jump-equivalent to $\gamma$ satisfying $d^{\Lambda}(\lambda(0)) < d^{\Lambda}(\gamma(0))$, we have
\[
{\rm length}^g(\lambda) \leq K\, {\rm length}^g(\gamma) < \infty.
\]
\end{proposition}

The proof of Proposition \ref{prop:ComparisonMonodromic} is given in subsection \ref{sec:ProofPropComparisonMonodromic} as a consequence of Theorem \ref{thm:RSgeneral}.  We give here just an idea of the proof.

\begin{remark}[Idea of the proof of Proposition \ref{prop:ComparisonMonodromic}]\label{rk:ProofPropComparisonMonodromic}\hfill
\begin{itemize}
\item[(i)] If $\gamma([0,1)) \cap \Sigma^1_{tr} = \emptyset$, then Proposition \ref{prop:ComparisonMonodromic} can be proved in a much more elemetary argument based on the following observation  via a geometrical argument.  Indeed, by properties (\ref{18sept1})-(\ref{eq:PoincareTransitionProp}) we note that, for all $k\in \N$ and all $p\in \Lambda$,
\[
d^{\Lambda}\left(\gamma(t_k^{\gamma})\right)>d^{\Lambda}(p) >d^{\Lambda} \left(\gamma(t_{k+1}^{\gamma})\right) \quad \implies \quad d^{\Lambda}\left(\gamma(t_{k+1}^{\gamma})\right)>d^{\Lambda}\left(T^{\Lambda}(p) \right) >d^{\Lambda} \left(\gamma(t_{k+2}^{\gamma})\right).
\]
So, if we denote by $\lambda_k$ the half-leaf of $\mathscr{L}$ connecting $p$ and $T^{\Lambda}(p)$, it follows by elementary (although non-trivial) geometrical arguments  that there exist $K>0$ and $\epsilon_k\geq 0$ such that
\[
{\rm length}^g(\lambda_k) \leq K\, {\rm length}^g \left(\gamma([t_k^{\gamma},t_{k+2}^{\gamma}])\right) + \epsilon_k \qquad \forall\,k \in \N,
\]
where the sequence $\{\epsilon_k\}$ is summable (since we will not use this fact, we do not prove it). This bound essentially allows one to prove \ref{prop:ComparisonMonodromic}, up to an extra additive constant in the bound ${\rm length}^g(\lambda) \leq K\, {\rm length}^g(\gamma)$ that anyhow is inessential for our purposes; note that this argument depends essentially on the fact that $\gamma(t_k^{\gamma})$ belongs to the same section $\Lambda$ for every $k$. 

\item[(ii)] In the case where $\gamma([0,1)) \cap \Sigma^1_{tr} \neq \emptyset$ is infinite, the situation is much more delicate. One needs to work with the countable composition of transition maps $T^{i_k}_{j_k}$ (in order to replace the Poincar\'{e} return), and the sequence of maps that one needs to consider is arbitrary. In particular, paths $\gamma$ whose jump sequences are non-periodic are specially challenging because we can not adapt the argument of the first part of the remark to this case. This justifies our use of more delicate singularity techniques (e.g. the regularity of transition maps \cite{speiss2} and the bi-Lipschitz class of the pulled-back metric \cite{HP}). This leads to the more technical statement in Theorem \ref{thm:RSgeneral}(IV) (see also Lemma \ref{lem:lengthSaddle}).
\end{itemize}
\end{remark}

We are now ready to prove Proposition  \ref{prop:MonoFiniteLenght}. 

\begin{proof}[Proof of Proposition \ref{prop:MonoFiniteLenght}]
Consider a monodromic trajectory  $\gamma:[0,1) \rightarrow \Sigma$ with limit $\bar{y} \in \widetilde{\Sigma}$ as above, and assume  that it has finite length. As before, we may assume that $\gamma$ is final, injective, and that $\bar{z}:=\gamma(0)\in \Lambda$. By Lemma \ref{lem:JumpsInterval}, for every $z\in \Lambda$ such that $d^{\Lambda}(z)<d^{\Lambda}(\bar{z}) $ there exists a unique final monodromic singular trajectory $\gamma^z :[0,1] \to \Sigma$, with $\gamma^z(0)=z$, which is jump-equivalent to $\gamma$. Moreover, by Proposition \ref{prop:ComparisonMonodromic} there exists $K>0$ such that 
\begin{eqnarray}\label{18sept2}
{\rm length}^g(\gamma^z) \leq K\, {\rm length}^g(\gamma).
\end{eqnarray}
Let $\{(i_k,j_k)\}_k$ be the sequence of jumps associated with $\gamma$. For every $z \in \Lambda$ with $d^{\Lambda}(z)<d^{\Lambda}(\bar{z}) $ the path $\gamma^z :[0,1] \to \Sigma$ is a singular horizontal path starting at $z$, so it admits a lift $\psi^z=(\gamma^z,p^z):[0,1] \rightarrow T^*M$ such that $\psi^z(0)=(z,p)$ with $p \perp \Delta_z$ and $|p|^*=1$ (see Proposition \ref{PROPsing}). Moreover, by (\ref{18sept2}) and Proposition \ref{PROPpbounded},
there exists $\tilde{K}>0$ such that
\begin{equation}
\label{eq:psi bounded}
|\psi^z(t)|^* \leq \tilde{K} \qquad \forall \, t \in [0,1], \quad \forall \,z \in \Lambda \mbox{ with }  d^{\Lambda}(z)<d^{\Lambda}(\bar{z}).
\end{equation}

Let $z \in \Lambda\cap \mathcal{V}$ such that $d^{\Lambda}(z)<d^{\Lambda}(\bar{z}) $ be fixed. Then there is  an injective  smooth path $\xi=(\alpha,\beta): [0,1] \rightarrow T^*M$ which satisfies the following properties:
\begin{eqnarray}\label{18june1}
\alpha(0)=\bar{z},\quad \alpha (1) = z, \quad \alpha(s) \in \Lambda, \quad \mbox{and} \quad d^{\Lambda}(\alpha(s)) \leq d^{\Lambda}(\alpha(s')) \qquad \forall \,0\leq s'\leq s \leq 1,
\end{eqnarray}
\begin{eqnarray}\label{18june2}
\beta(s) \perp \Delta_{\alpha(s)} \quad \mbox{and} \quad |\beta(s)|^*=1 \qquad \forall\, s \in [0,1],
\end{eqnarray}
and
\begin{eqnarray}\label{18june3}
A:= \int_{0}^1 \beta(s) \cdot \dot{\alpha}(s) \, ds >0.
\end{eqnarray}
Note that (\ref{18june3}) can be satisfied because $\Delta$ is transverse to $\Lambda$. For every $s\in [0,1]$, set $\gamma^s:=\gamma^{\alpha(s)}$ and note that $\gamma^0=\gamma$. By construction, each path $\gamma^s$ has the same sequence of jumps $\{(i_k,j_k)\}_{k\in \N}$ which is associated to sequences of times $\{t^s_k:=t_k^{\gamma^s}\}_{k\in \N}$. For every $s\in [0,1]$, denote by $\psi^s=(\gamma^s,p^s)$ the abnormal lift associated to $\gamma^s$ starting at $(\alpha(s),\beta(s))=\xi(s)$. We may assume without loss of generality that $p^s=p^{\alpha(s)}$ for all $s\in [0,1]$, so that $\psi^s=\psi^{\alpha(s)}$. 

\begin{figure}\label{fig5}
\begin{center}
\includegraphics[width=4.7cm]{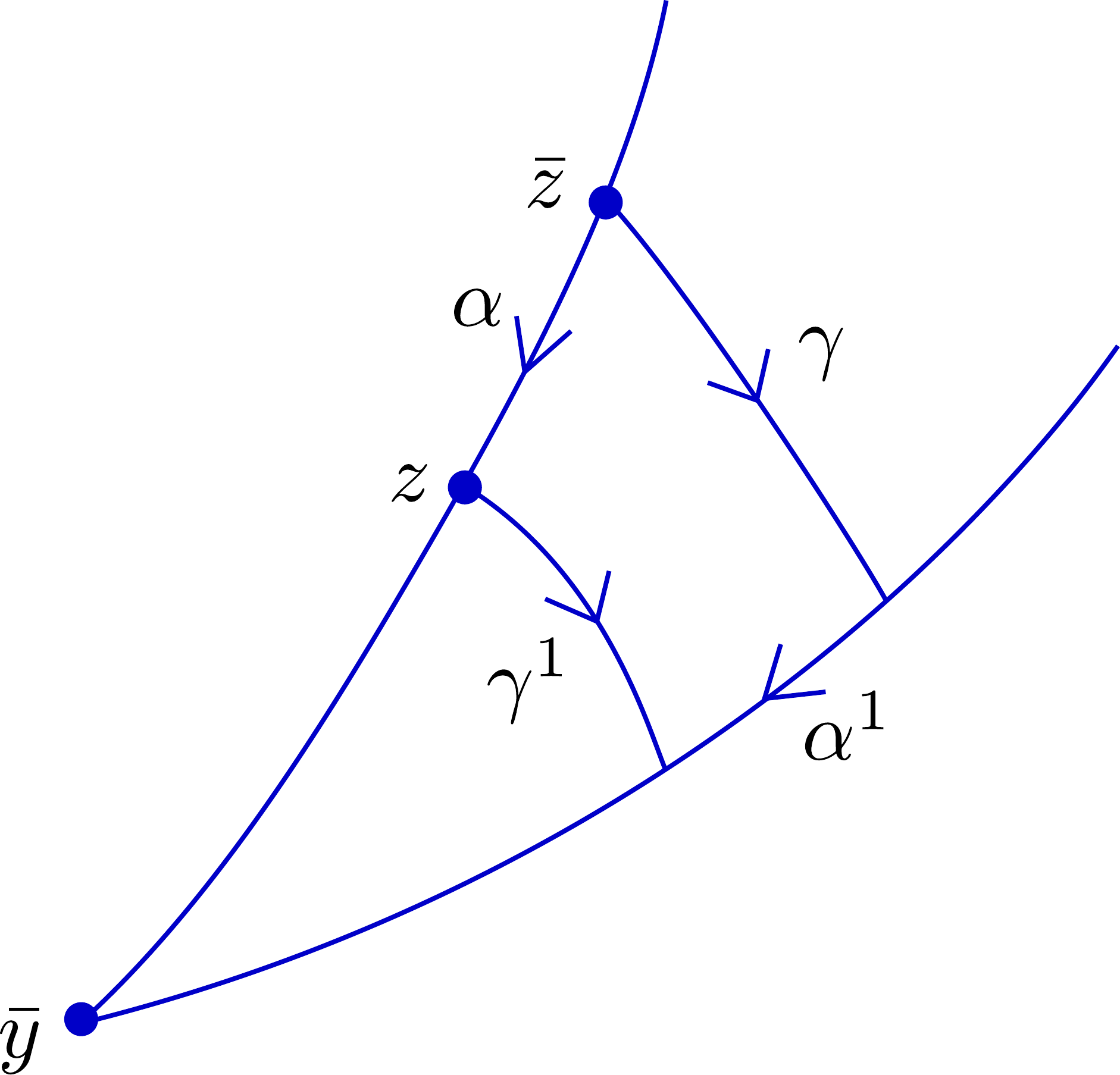}
\caption{The projection of $\mathcal{S}^0$ in $M$}
\end{center}
\end{figure}

From $t_0=0$ to $t_1^s$, the set of lifts $\psi^s=(\gamma^s,p^s):=\psi^{\alpha(s)}$ of the paths $\gamma^s$ starting at $(\alpha^s,\beta^s)$ span a surface $\mathcal{S}^0$ homeomorphic to a 2-dimensional disc  whose boundary is composed by $\xi$,  the lift $\psi^0|_{[0,t_1^0]}$, the lift $\psi^1|_{[0,t_1^1]}$,
and a path $\xi^1=(\alpha^1,\beta^1):[0,1] \rightarrow T^*M$ whose projection is contained in $\Gamma^{J(i_0,j_0)}$
and which connects $\psi^0(t_1^0)$ to $\psi^1(t_1^1)$ (see Figure 5).  Thus, by Stokes' Theorem we have
$$
\int_{\mathcal{S}^0} \omega = \int_{0}^1 \beta(s) \cdot \dot{\alpha}(s) \, ds +  \int_{0}^{t_1^1} p^1(t) \cdot \dot{\gamma}^1(t) \, dt - \int_{0}^1 \beta^1(s) \cdot \dot{\alpha}^1(s) \, ds -  \int_{0}^{t_1^0} p^0(t) \cdot \dot{\gamma}(t) \, dt.
$$
Since $\gamma$ and $\gamma^1$ are both singular horizontal paths we have $p^0(t) \cdot \dot{\gamma}(t) =  p^1(t) \cdot \dot{\gamma}^1(t)=0$ for all $t\in [0,1]$. Moreover, since the derivative of the lifts $\psi^s$ is always contained in the kernel of $\omega_{\vert \Delta^{\perp}}$ (see Appendix \ref{SECSingular}), we have $\int_{\mathcal{S}^0} \omega=0$. As a consequence, we infer that
$$
\int_{0}^1 \beta^1(s) \cdot \dot{\alpha}^1(s) \, ds = \int_{0}^1 \beta(s) \cdot \dot{\alpha}(s) \, ds =A.
$$
Repeating this argument and recalling \eqref{eq:psi bounded}, we get a sequence of arcs $\xi^k=(\alpha^k,\beta^k):[0,1] \rightarrow T^*M$ such that 
$$
\int_{0}^1 \beta^k(s) \cdot \dot{\alpha}^k(s) \, ds =A>0 \qquad \forall\, k \in \N,
$$ 
and 
$$
{\rm length}^g(\alpha^k) \rightarrow 0 \quad \mbox{and} \quad |\beta^k(s)|^* \leq \tilde{K} \qquad \forall \, k \in \N.
$$
This provides the desired contradiction, which proves the result.

\end{proof}

The proof of Theorem \ref{THMSard}  is given hereafter  as a consequence of both Proposition \ref{prop:TopDichSing} and Proposition \ref{prop:MonoFiniteLenght}.

\subsection{Proof of Theorem \ref{THMSard}}

Before starting the proof let us summarize the different types of points $y\in \Sigma$ that can be crossed by a singular horizontal path. 
We distinguish four cases. \\

\noindent First case: $y\in \Sigma^2$.\\
The line foliation is regular in a neighborhood of $y$, so there is an analytic curve such that any singular path containing $y$ is locally contained in this curve. \\

\noindent Second case: $y\in \widetilde{\Sigma}^1_{tan}:=\Sigma^1_{tan}\setminus \widetilde{\Sigma}^0$.\\
By Proposition \ref{prop:TopDichSing} and the fact that $\widetilde{\Sigma}^0$ is locally finite, any singular path passing through $y$ is contained in $\widetilde{\Sigma}^1_{tan}$, that is locally analytic. \\

\noindent Third case: $y\in \widetilde{\Sigma}^0$.\\
The singular paths that contain $y$ are either the branches of ${\Sigma}^1_{tan}$ or the characteristic singular paths. In the first case, these branches are actually contained inside $\widetilde \Sigma^1_{tan}$ with the exception of $y$. In the second case, there are only finitely many characteristic singular paths by Proposition \ref{prop:TopDichSing}, and they are semianalytic by Proposition \ref{prop:characteristic-semianalytic}.  
 \\

\noindent Fourth case: $y\in \Sigma^1_{tr}$.\\
By Lemma \ref{rk:LocalTriviality}, there are finitely many semianalytic  singular horizontal curves that can cross $y$. \\

In conclusion, if we travel along a given singular path $\gamma:[0,1]\rightarrow M$ then bifurcation points may happen only when $\gamma$ crosses the set 
$\widetilde{\Sigma}^0 \cup \Sigma^1_{tr}$. Since $\mathcal{X}^{x,L}_{\Delta,g}$ 
is compact, there are only finitely many points of $\widetilde{\Sigma}^0$ to consider. Moreover, by Lemma \ref{rk:LocalTriviality},  from every bifurcation point in $\Sigma^1_{tr}$  there are only finitely many curves exiting from it. 
By Proposition \ref{prop:TopDichSing} any singular horizontal path interesect $\Sigma^1_{tr}$ finitely many times, but we need to  show that 
the intersection of $\mathcal{X}^{x,L}_{\Delta,g}$ with $\Sigma^1_{tr}$ is finite.  This follows from the fact that $\mathcal{X}^{x,L}_{\Delta,g}$ can be 
constructed from finitely many singular path emanating from $x$, by successive 
finite branching at the points of $\widetilde{\Sigma}^0 \cup \Sigma^1_{tr}$ met by 
the paths.  Let us present this argument precisely.

We associate to $\mathcal{X}^{x,L}_{\Delta,g}$ a tree $T$ constructed 
recursively as follows.  Let the initial vertex $v_0$ of the tree  
represent the point $x$ and let the edges from $v_0$ be in one-to-one correspondence with different singular horizontal paths starting from $x$.  
If such path arrives to a branching point, that is a point of $\widetilde\Sigma_0\cup \Sigma^1_{tr}$, we represent this point as another vertex of the tree 
(even if this point is again $x$).  If a singular path does not arrive at 
$\widetilde\Sigma_0\cup \Sigma^1_{tr}$ we just add formally a (final) vertex.  
In this way we construct  a connected (a priori infinite) locally finite tree.  
We note that any injective singular horizontal path starting at $x$, 
of length bounded by $L$, is represented in $T$ by a finite simple path of the tree (a path with no repeated vertices).  

Suppose, by contradiction, that $T$ is infinite.  
By K\"onig's Lemma (see, {\it e.g.} \cite{wilson96}), the tree $T$ contains a simple path $\omega_\infty$ that starts at $v_0$ and continues from it through infinitely many vertices.  Such path corresponds to a singular horizontal trajectory 
$\gamma_\infty$ that passes infinitely many times through $\widetilde\Sigma_0\cup \Sigma^1_{tr}$. Since any finite subpath of $\omega_\infty$ corresponds to 
a singular horizontal path of length bounded by $L$, $\gamma_\infty$ itself has length bounded by $L$ and crosses infinitely many times $\Sigma^1_{tr}$ (a finite length path cannot pass infinitely many times through 
$\widetilde\Sigma_0\cap \mathcal{X}^{x,L}_{\Delta,g}$ that is finite). 
Hence:\\
- either  $\gamma_\infty$ is monodromic of finite length, and this  
 contradicts Proposition \ref{prop:TopDichSing};\\
 - or the limit point of  $\gamma_\infty$ belongs to $\Sigma^1_{tr}$, which contradicts Lemma \ref{rk:LocalTriviality}. \\
  Therefore, the tree $T$ is finite, and $\mathcal{X}^{x,L}_{\Delta,g}$ consist of finitely many singular horizontal curves.  \\

The last part of Theorem \ref{THMSard} follows from the fact that any smooth manifold can be equipped with a complete Riemannian metric (see \cite{no60}).

\section{Singularities of the characteristic line-foliation}\label{sec:SingCharc}

\subsection{Divergence property}\label{sec:divergence}

In this subsection we introduce some basic results about the divergence of vector fields. The subsection follows a slightly more general setting than the previous section, but which relates to the study of the Sard Conjecture via the local model given in Remark \ref{rk:LocalDescriptMartinet}(i).

We start by considering a nonsingular analytic surface ${\mathscr S}$ with a volume form $\omS$. Denote by $\mathcal{O}_{\mathscr S}$ the sheaf of analytic functions over ${\mathscr S}$. We note that there exists a one-to-one correspondence between $1$-differential forms $\eta \in \Omega^1({\mathscr S})$ and vector fields $\mathcal Z \in Der_{\mathscr S}$ given by
$$
\mathcal Z \longleftrightarrow \eta \qquad \text { if } \eta = i_{{\mathcal Z}}\omS.
$$
This correspondence gives the following formula on the divergence:
\[
\div_{\omS}({\mathcal Z})  \omS = d\eta .
\]
\begin{remark}[Basic properties]\label{rk:localDivergence}\hfill
\begin{itemize}
\item[(i)]
Suppose that $u,v$ are local coordinates on ${\mathscr S}$ such that $\omS = du\wedge dv$. Then the form 
$\eta = \alpha du + \beta dv$ corresponds to $\mathcal Z = \alpha \partial_y - \beta \partial_x$. 
\item[(ii)] Given an analytic function $f: S \to \mathbb{R}$, it holds 
 $$df\wedge \eta = df\wedge  i_{{\mathcal Z}} \omS =  i_{{\mathcal Z}} df \wedge \omS = \mathcal Z(f) \omS .$$
 \item[(iii)] The above results can be easily generalized to $d$-dimensional analytic manifolds, where the one-to-one correspondence is between $d-1$ forms and vector fields (that is, between $\Omega^{d-1}(M)$ and $Der_M$).
 \end{itemize}
\end{remark}

We denote by ${\mathcal Z}(\mathcal{O}_{\mathscr S})$ the ideal sheaf generated by the derivation ${\mathcal Z}$ applied to the analytic functions in $\mathcal{O}_{\mathscr S}$, that is, the ideal sheaf locally generated by the coefficients of ${\mathcal Z}$. In what follows, we  study closely the property $\div_{\omS}(\mathcal Z) \in \mathcal Z(\mathcal{O}_{\mathscr S})$, following \cite[Lemma 2.3 and 3.2]{BR}. The next result shows that the property is independent of the volume form.

\begin{lemma}[Intrinsicality]
Let $\omS$ and $\omS'$ be two volume forms over $\mathscr S$. Then $\div_{\omS}(\mathcal{Z}) \in \mathcal{Z}(\mathcal{O}_{\mathscr S})$ if and only if $\div_{\omS'}(\mathcal{Z})  \in \mathcal{Z}(\mathcal{O}_{\mathscr S})$.
\label{lem:intrinsic}
\end{lemma}
\begin{proof}
Given a point $p \in \mathscr S$, there exists an open neighborhood $U$ of $p$ and a smooth function $F: U \to \mathbb{R}$ which is everywhere non-zero and such that $\omS' = F\cdot  \omS$ in $U$. Therefore,
\[
\begin{aligned}
\div_{\omS'}(\mathcal{Z}) \cdot  \omS' &= d\left( i_{\mathcal{Z}}\omS'  \right) = d\left( F \cdot  i_{\mathcal{Z}}\omS  \right) =  \left[\mathcal{Z}(F)/F  + \div_{\omS}(\mathcal{Z}) \right]\omS',
\end{aligned}
\]
and we conclude easily. 
\end{proof}

Lemma \ref{lem:final} below illustrates the importance of this property; in its statement we use the notion of elementary singularities (see, e.g. \cite[Definition 4.27]{IY})), that we recall in Appendix \ref{app:ReductionLineFoliations} (Definition \ref{def:ElementarySing}). 

\begin{lemma}[Final Singularities]\label{lem:final}
Let $\mathcal{Z}$ be a real analytic vector-field defined in an open neighborhood $U \subset \mathbb{R}^2$ of the origin and $\omega_U$ to be a volume form over $U$. Let $(x,y)$ be a coordinate system defined over $U$ and suppose that: 
\begin{itemize}
\item[(i)] $\div_{\omega_U}(\mathcal{Z}) \in \mathcal{Z}(\mathcal{O}_{U})$; 
\item[(ii)] $\mathcal{Z} = x^{\alpha}y^{\beta}\, \widetilde{\mathcal{Z}}$, for some $\alpha$ and $\beta \in \mathbb{N}$, where $\widetilde{\mathcal{Z}}$ is either regular, or its singular points are isolated elementary singularities.
\end{itemize}
Then the vector field $\widetilde{\mathcal{Z}}$ is tangent to the set $\{x^{\alpha}y^{\beta}=0\}$ and all of its singularities are saddles.
\end{lemma}
\begin{proof}[Proof of Lemma \ref{lem:final}]
By Lemma \ref{lem:intrinsic}, up to shrinking $U$ we can suppose that $\omega_U = dx\wedge dy$. We denote by $A =\mathcal{Z}(x)$ and $B=\mathcal{Z}(y)$. By assumption (ii), these functions are divisible by $x^{\alpha}y^{\beta}$, namely $A=x^{\alpha}y^{\beta} \widetilde{A}$ and $B=x^{\alpha}y^{\beta}\widetilde{B}$. By assumption (i), there exist smooth functions $f$ and $g$ such that
\begin{equation}\label{eq:1}
\partial_xA + \partial_yB = f\cdot A + g\cdot B \quad \text{ and } \quad \alpha \cdot \widetilde{A}/x + \partial_x\widetilde{A} +  \beta\cdot \widetilde{B}/y + \partial_y\widetilde{B} = f\cdot \widetilde{A} + g\cdot \widetilde{B}
\end{equation}
In particular  $\alpha \widetilde{A}/x + \beta \widetilde{B}/y $ does not have poles, which implies that $\widetilde{A}$ is divisible by $x$ if $\alpha \neq 0$, and $\widetilde{B}$ is divisible by $y$ if $\beta \neq 0$. In other words, $\widetilde{\mathcal{Z}}$ is tangent to $\{x^{\alpha}y^{\beta}=0\}$.

Without loss of generality, we can suppose that the origin is the only singularity of $\widetilde{\mathcal{Z}}$. We consider the determinant and the trace of the Jacobian of $\widetilde{\mathcal{Z}}$ at the origin:
\[
\begin{aligned}
\det\bigl(\mbox{Jac}(\widetilde{\mathcal{Z}})(0)\bigr)&= \partial_x\widetilde{A}(0)\cdot \partial_y\widetilde{B}(0) - \partial_y\widetilde{A}(0)\cdot \partial_x\widetilde{B}(0),\\
\mbox{tr}\bigl(\mbox{Jac}(\widetilde{\mathcal{Z}})(0)\bigr)&= \partial_x\widetilde{A}(0)+ \partial_y\widetilde{B}(0).
\end{aligned}.
\]
In order to conclude, thanks to Remark \ref{rk:ElembentarySing}(i) it is enough to prove that $\det\bigl(\mbox{Jac}(\widetilde{\mathcal{Z}})(0)\bigr)<0$. We divide in two cases, depending on the value of $\alpha$ and $\beta$.

First, suppose that $\alpha=\beta=0$ (in particular $A = \widetilde{A}$ and $B=\widetilde{B}$). Then, thanks to \eqref{eq:1},
\[
\mbox{tr}\bigl(\mbox{Jac}(\widetilde{\mathcal{Z}})(0)\bigr) = \partial_xA(0) + \partial_yB(0) = 0.
\]
Since the origin is an elementary singularity of $\widetilde{\mathcal{Z}}$, using Remark \ref{rk:ElembentarySing}(ii) we conclude that the determinant is negative. Thus, the singularity is a saddle point.

Next, without loss of generality we suppose that $\alpha \neq 0$. In this case $x$ divides $\widetilde{A}$, which implies that $\partial_y\widetilde{A}(0)=0$ and $\partial_x\widetilde{A}(0)= \bigl(\widetilde{A}/x\bigr)(0)$. In particular, this yields
\begin{equation}
\label{eq:det}
\det\bigl(\mbox{Jac}(\widetilde{\mathcal{Z}})(0)\bigr) = \partial_x\widetilde{A}(0)\cdot \partial_y\widetilde{B}(0).
\end{equation}
Also, since $\partial_x\widetilde{A}(0)= \bigl(\widetilde{A}/x\bigr)(0)$, and either $\beta=0$ or $\partial_y\widetilde{B}(0)= \bigl(\widetilde{B}/y\bigr)(0)$, using \eqref{eq:1} we get
\[
(\alpha+1)\partial_x\widetilde{A}(0) + (\beta+1)\partial_y\widetilde{B}(0)=0.
\]
It follows that $\partial_x\widetilde{A}(0)$ and $\partial_y\widetilde{B}(0)$ have opposite signs (if they are both zero then the determinant and the trace are zero, contradicting the definition of elementary singularity), and therefore the determinant is negative (see \eqref{eq:det}). Once again, since the origin is an elementary singularity of $\widetilde{\mathcal{Z}}$, using  Remark \ref{rk:ElembentarySing}(ii) we conclude that the singularity is a saddle point.
\end{proof}

Next, suppose that $M$ is a 3-dimensional analytic manifold and denote by $\omega_M$ its volume form. We now start the study over the Martinet surface $\Sigma$, cf. Remark \ref{rk:LocalDescriptMartinet}(i).

Let $\delta\in \Omega^1(M)$ be an everywhere non-singular analytic $1$-form and denote by $h$ the analytic function defined as in equation \eqref{eq:martinetequation}. Denote by $*: \Omega^1 (M) \to \Omega ^2(M)$ the Hodge star operator, cf. \cite[Ch. V]{wells}. We start by a known characterization of $d\delta$ in terms of $\delta$ and $*\delta$:

\begin{lemma}\label{lem:DivFirstFormula}
There exists an analytic form $a\in \Omega^1 (M)$ such that:
\begin{align}\label{eq:decomposition2}
d\delta  = a\wedge \delta + h \langle \delta , *\delta \rangle^{-1} (*\delta).  
\end{align}
\end{lemma}
\begin{proof}[Proof of Lemma \ref{lem:DivFirstFormula}]
Since $\delta$ is nonsingular, the induced scalar product $\langle \delta , *\delta \rangle$ is a nowhere vanishing function  and we have 
\begin{equation}\label{eq:scalarproduct}
\delta \wedge *\delta  = \langle \delta , *\delta \rangle \omM.
\end{equation}
The form $d\delta$ can be decomposed  as
\begin{align}\label{eq:decomposition}
d\delta  = a\wedge \delta + b (*\delta)
\end{align}
where ${a}$ is an analytic $1$-form and $ b$  is an analytic function. Combining \eqref{eq:martinetequation}, \eqref{eq:scalarproduct}, and \eqref{eq:decomposition}, we deduce that
$h=  b \langle \delta , *\delta \rangle $, which proves \eqref{eq:decomposition2}.
\end{proof}

Now, we consider an analytic map $\pi: \mathscr S\to \Sigma \subset M$ from an analytic surface $\mathscr S$ to the Martinet surface $\Sigma$, and we set $\eta := \pi^* (\delta)$. It follows from Lemma \ref{lem:DivFirstFormula} that
\begin{align}\label{eq:relation}
d\eta = \tilde a \wedge \eta
\end{align}
with $\tilde a = \pi^* a$. Let $\mathcal Z$ be the vector field associated to $\eta$, and denote by $\mathcal{Z}(\pi)$ the ideal subsheaf of $\mathcal{Z}(\mathcal{O}_{\mathscr S})$ generated  by the derivation $\mathcal{Z}$ applied to the pullback by $\pi$ of analytic functions on $M$.  

\begin{remark}\label{rk:BasicRemarksPi}\hfill
\begin{itemize}
\item[(i)] For our applications, the map $\pi$ is either going to be an inclusion of the regular part of $\Sigma$ into $M$, or a resolution of singularities of (the analytic space) $\Sigma$ (cf. Theorem \ref{thm:RSgeneral}). 
\item[(ii)] If we write (locally) $\pi=(\pi_1, \pi_2, \pi_3)$, then $\mathcal{Z}(\pi)$ is locally generated by 
$\mathcal Z (\pi_1), \mathcal Z (\pi_2), \mathcal Z (\pi_3)$. 
\end{itemize}
\end{remark}

The next proposition shows that, in the local setting (following Remarks \ref{rk:LocalDescriptMartinet}(i) and \ref{rk:BasicRemarksPi}(i)), the property $\div_{\omS}({\mathcal Z}) \in {\mathcal Z}(\mathcal{O}_{\mathscr S})$ is always satisfied. This can be seen as a reformulation of \cite[Lemmas 2.3, 3.1, and 4.3]{BR}

\begin{proposition}[Divergence bound]\label{prop:DivProperty} 
Let $\eta \in \Omega^1(\mathscr S)$, and let $\mathcal Z$ be the vector field associated to $\eta$.
\begin{itemize}
\item[(i)] If $\eta \in \Omega^1(\mathscr S)$ satisfies \eqref{eq:relation}, then 
 $\div_{\omS}({\mathcal Z}) \in {\mathcal Z}(\mathcal{O}_{\mathscr S})$. 
\item[(ii)] If in addition $\eta = \pi^* (\delta)$,  then $\div_{\omS}({\mathcal Z}) \in \mathcal{Z}(\pi)$. In particular,  for every compact subset $\mathcal K\subset \mathscr S$  there is a constant $K>0$ such that
$$
\left|\div_{\omS}({\mathcal Z}) \right| \le K\left|\pi_* (\mathcal Z)\right|\qquad \text{ on $\mathcal K$}.
$$
\end{itemize}
\end{proposition}

\begin{proof}[Proof of Proposition \ref{prop:DivProperty}]
Let  ${a}$ be as in \eqref{eq:decomposition},
and write it in local coordinates on $M$ 
as $a = \sum g_i dx_i$. Then $\pi^* (a)  = \sum (g_i\circ \pi)\,   d\pi_i$, which implies that, in local 
coordinates on $\mathscr S$, we have
$$
\div_{\omS}({\mathcal Z})  \omS = \sum (g_i\circ \pi)  d\pi_i\wedge \eta =  \sum  (g_i\circ \pi) \mathcal Z (\pi_i)  \omS.
$$
The bound follows from the fact that $\pi_* (\mathcal Z) = ( \mathcal Z (\pi_1), \mathcal Z (\pi_2), \mathcal Z (\pi_3))$.  
\end{proof}

\subsection{Resolution of singularities}\label{sec:ResSing}

Here we follow the notation and framework introduced in Appendix \ref{app:ResolutionOfSingularities} and in subsections \ref{sec:framework} and \ref{sec:CharacteristicFoliation}. All definitions and concepts concerning resolution of singularities (e.g. blowings-up, simple normal crossing divisors, strict transforms, etc) are recalled in Appendix \ref{app:ResolutionOfSingularities}.

\begin{theorem}\label{thm:RSgeneral}
There exist an analytic surface $\mathscr{S}$, and a simple normal crossing divisor $E$ over $\mathscr{S}$ (that is, a locally finite union of irreducible smooth divisors, see subsection \ref{SSEC:SNC}), and a proper analytic morphism $\pi : \mathscr{S} \to \Sigma$ (a sequence of admissible blowings-up with exceptional divisor $E$, see Definition \ref{def:SeqAdmBlowings-up}) such that:
\begin{itemize}
\item[(I)] The restriction of $\pi$ to $\mathscr{S}\setminus E$ is a diffeomorphism onto its image $\Sigma \setminus S$ (c.f. Lemma \ref{lem:CharacteristicFoliation}).

\item[(II)] Denote by $\widetilde{\mathscr{L}}$ the strict transform of the foliation $\mathscr{L}$ (cf. subsection \ref{app:ReductionLineFoliations}). Then all singularities of $\widetilde{\mathscr{L}}$ are saddle points.

\item[(III)] The exceptional divisor $E$  is given by the union of two locally finite sets of divisors $E_{tan}$ and $E_{tr}$, where $E_{tan} \cap E_{tr}$ is a locally finite set of points, such that $\widetilde{\mathscr{L}}$ is tangent to $E_{tan}$ and $\widetilde{\mathscr{L}}$ is everywhere transverse to $E_{tr}$. Furthermore, the log-rank of $\pi$ over $E_{tr}\setminus E_{tan}$ is constant equal to $1$ (we recall the definition of log-rank in Appendix \ref{app:HP}).

\item[(IV)] At each point $\bar{z}\in E_{tan}$, there exists an open neighborhood $U_{\bar{z}}$ of $\bar{z}$ such that: 
\begin{itemize}
\item[(i)] Suppose that there exists only one irreducible component of $E_{tan}$ passing through $\bar{z}$. Then there exists a coordinate system $(u,v)$ centered at $\bar{z}$ and defined over $U_{\bar{z}}$, such that:
\begin{itemize}
\item[(a)] The exceptional divisor $E_{tan}$ restricted to $U_{\bar{z}}$ coincides with $\{u=0\}$.

\begin{figure}\label{fig6}
\begin{center}
\includegraphics[width=6cm]{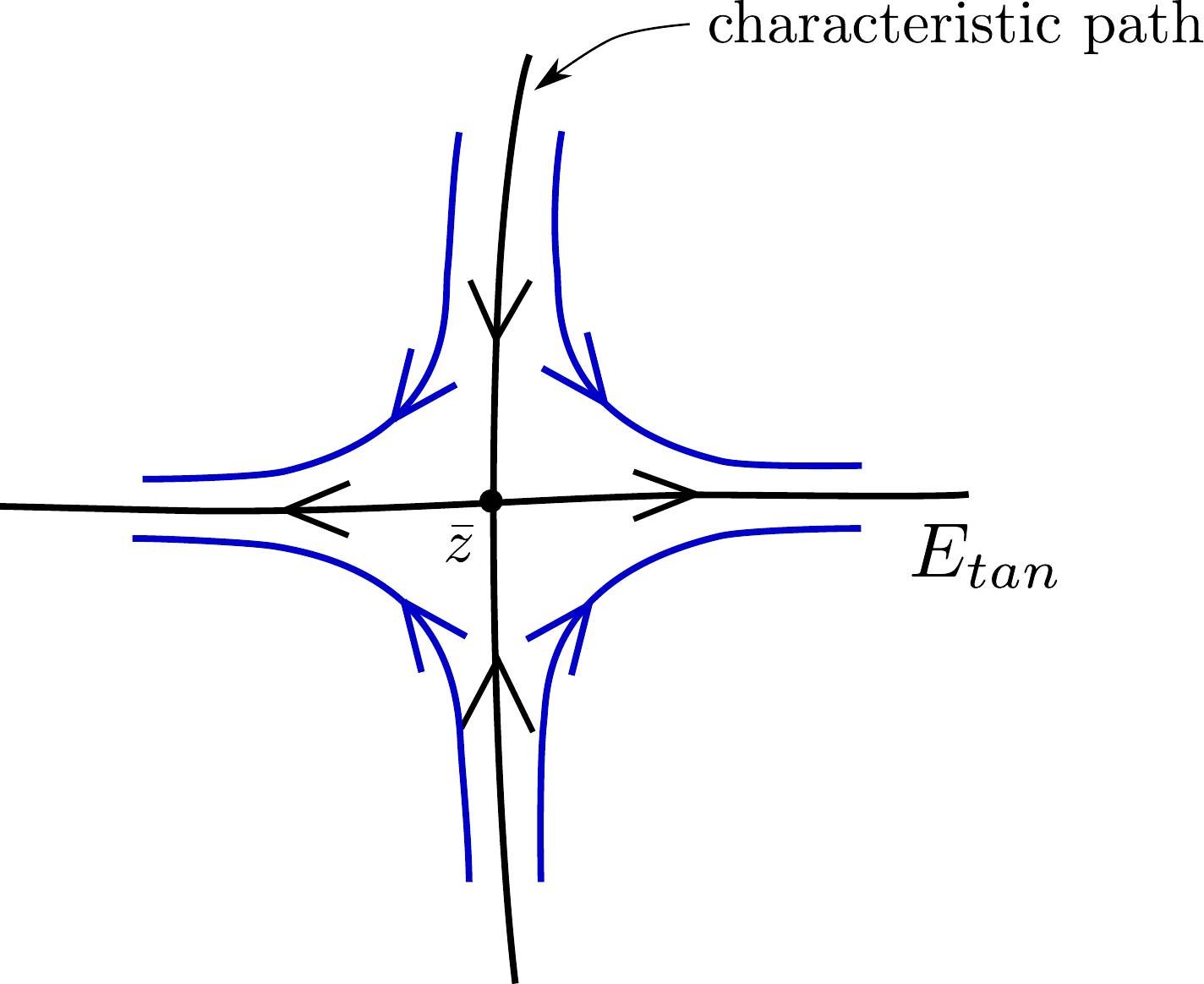}
\caption{A saddle point as in the first case of Theorem \ref{thm:RSgeneral}(IV.i.b)}
\end{center}
\end{figure}

\item[(b)] Either $\bar{z}$ is a saddle point of $\widetilde{\mathscr{L}}$ (see Figure 6); or at each half-plane (bounded by $E_{tan}$) there exist two smooth analytic semi-segments $\Lambda_{\bar{z}}^1$ and $\Lambda_{\bar{z}}^2$ which are transverse to $\widetilde{\mathscr{L}}$ and $E_{tan}$, such that the flow (of a local generator ${\widetilde{\mathcal{Z}}}$) associated to $\widetilde{\mathscr{L}}$ gives rise to a bi-analytic transition map
\[
\phi_{\bar{z}} : \Lambda_{\bar{z}}^1 \to \Lambda_{\bar{z}}^2,
\]
and there exists a rectangle $V_{\bar{z}}$ bounded by $E_{tan}$, $\Lambda_{\bar{z}}^1$, $\Lambda_{\bar{z}}^2$ and a regular leaf $\mathcal{L} \not\subset E_{tan}$ of $\widetilde{\mathscr{L}}$ such that $\bar{z} \in \partial V_{\bar{z}} \setminus (\Lambda_{\bar{z}}^1 \cup \Lambda_{\bar{z}}^2 \cup \mathcal{L})$ (see Figure 7).

\begin{figure}\label{fig7}
\begin{center}
\includegraphics[width=5cm]{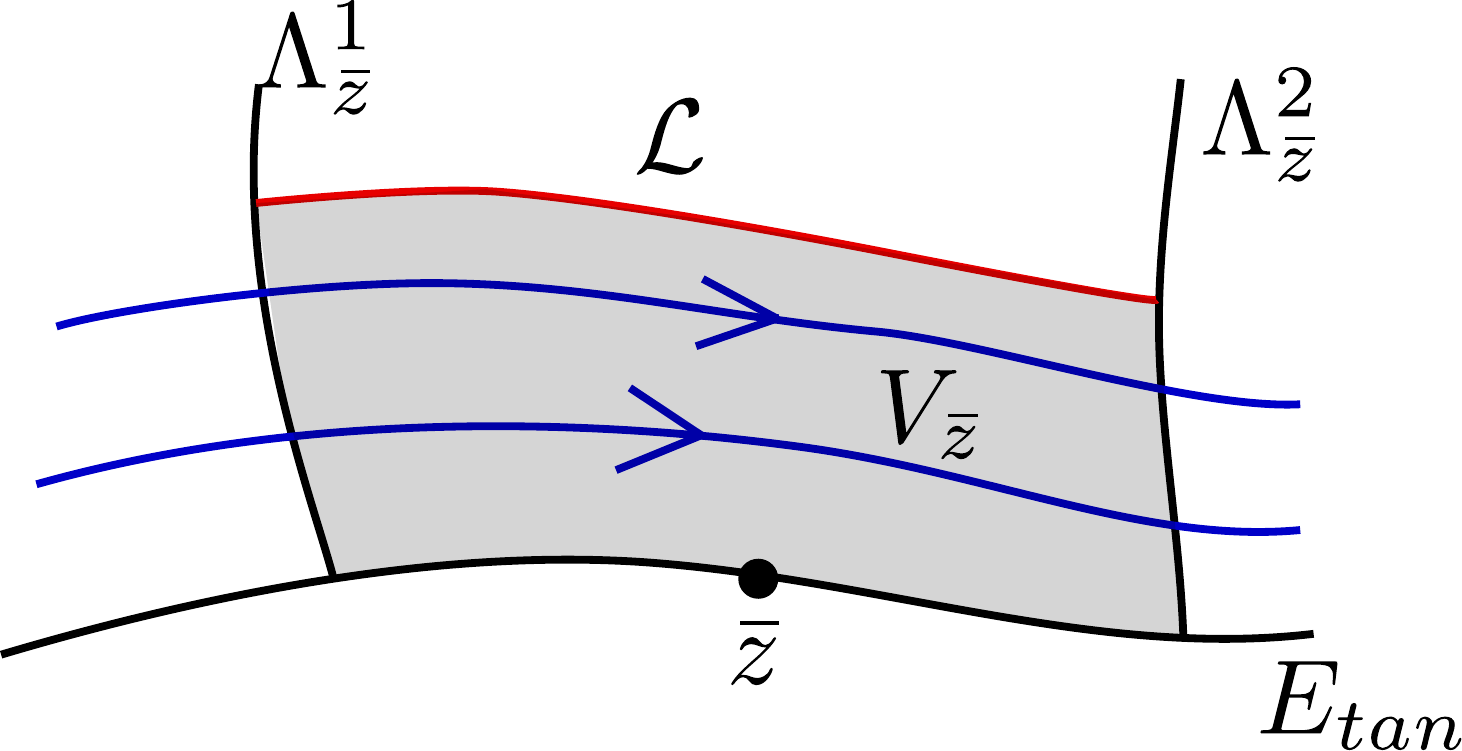}
\caption{Second case of Theorem \ref{thm:RSgeneral}(IV.i.b)}
\end{center}
\end{figure}
\begin{figure}\label{fig7bis}
\begin{center}
\includegraphics[width=5cm]{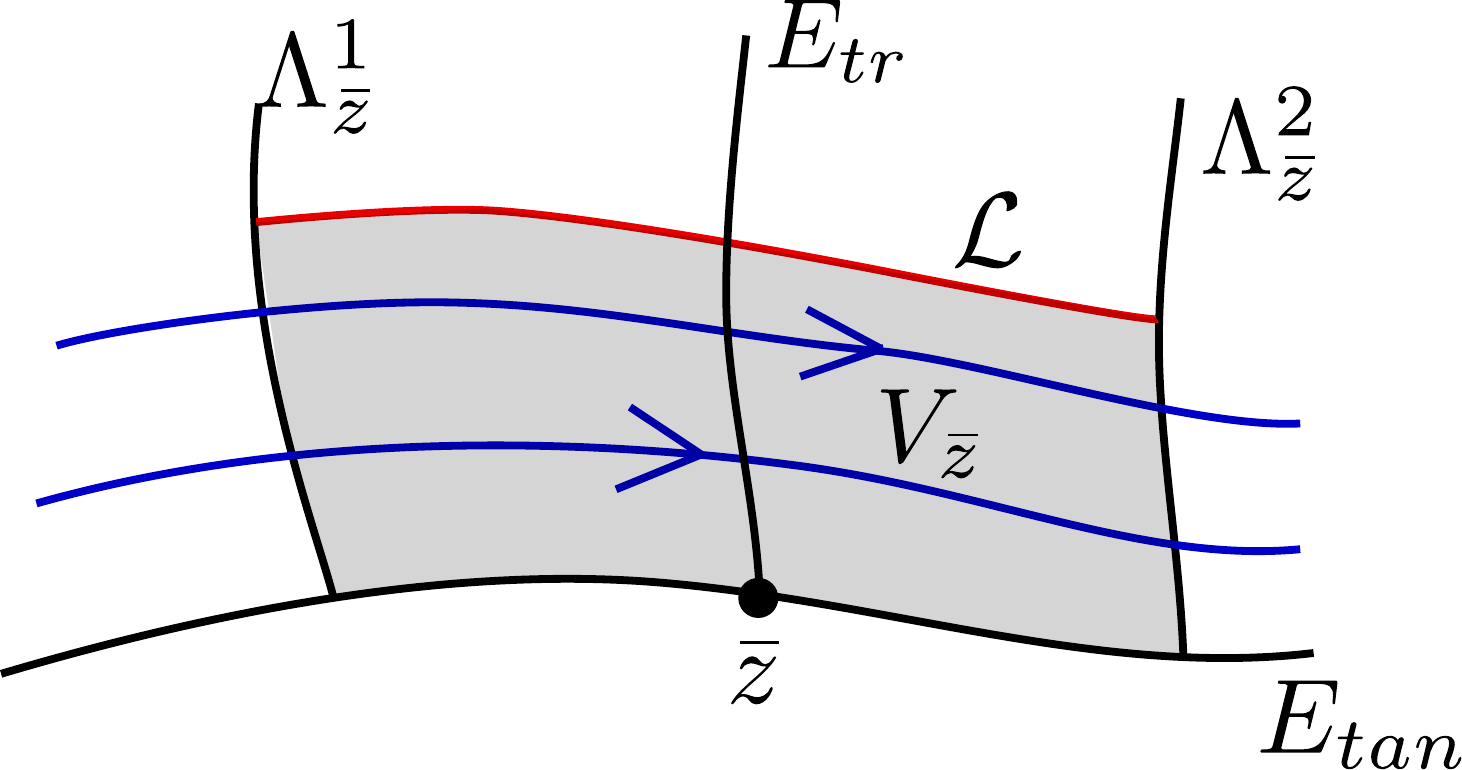}
\caption{Case  (IV.i.c) in Theorem \ref{thm:RSgeneral}}
\end{center}
\end{figure}


\item[(c)] If $\bar{z} \in E_{tr}$, then $\bar{z}$ is a regular point of $\widetilde{\mathscr{L}}$ and $E_{tr} \cap U_{\bar{z}} = \{v=0\}$ does not intersect $\Lambda_{\bar{z}}^1$ nor $\Lambda_{\bar{z}}^2$. Furthermore, the map $\phi_{\bar{z}}$ is the composition of two analytic maps  (see Figure 8):
\[
\phi_{\bar{z}}^1: \Lambda_{\bar{z}}^1 \to  E_{tr}, \qquad \phi_{\bar{z}}^2: E_{tr} \to \Lambda_{\bar{z}}^2.
\]
\end{itemize}

\item[(ii)] Suppose that there exists two irreducible components of $E_{tan}$ passing through $\bar{z}$. Then there exists a coordinate system $\pmb{u}=(u_1,u_2)$ centered at $\bar{z}$ and defined over $U_{\bar{z}}$, such that:
\begin{itemize}
\item[(a)] The exceptional divisor $E_{tan}$ restricted to $U_{\bar{z}}$ coincides with $\{u_1\cdot u_2=0\}$.
\item[(b)] At each quadrant (bounded by $E_{tan}$) there exists two smooth analytic semi-segments $\Lambda_{\bar{z}}^1$ and $\Lambda_{\bar{z}}^2$ which are transverse to $\widetilde{\mathscr{L}}$ and to $E_{tan}$, such that the flow (of a local generator $\widetilde{\mathcal{Z}}$) associated to $\widetilde{\mathscr{L}}$ gives rise to a bijective (but not necessarily analytic) 
transition map
\[
\phi_{\bar{z}} : \Lambda_{\bar{z}}^1 \to \Lambda_{\bar{z}}^2
\]
and there exists a rectangle $V_{\bar{z}}$ bounded by $E_{tan}$, $\Lambda_{\bar{z}}^1$, $\Lambda_{\bar{z}}^2$ and a regular leaf $\mathcal{L} \not\subset E$ of $\widetilde{\mathscr{L}}$ such that $\bar{z} \in \partial V_{\bar{z}} \setminus (\Lambda_{\bar{z}}^1 \cup \Lambda_{\bar{z}}^2 \cup \mathcal{L})$ (see Figure 9).
\begin{figure}\label{fig8}
\begin{center}
\includegraphics[width=5cm]{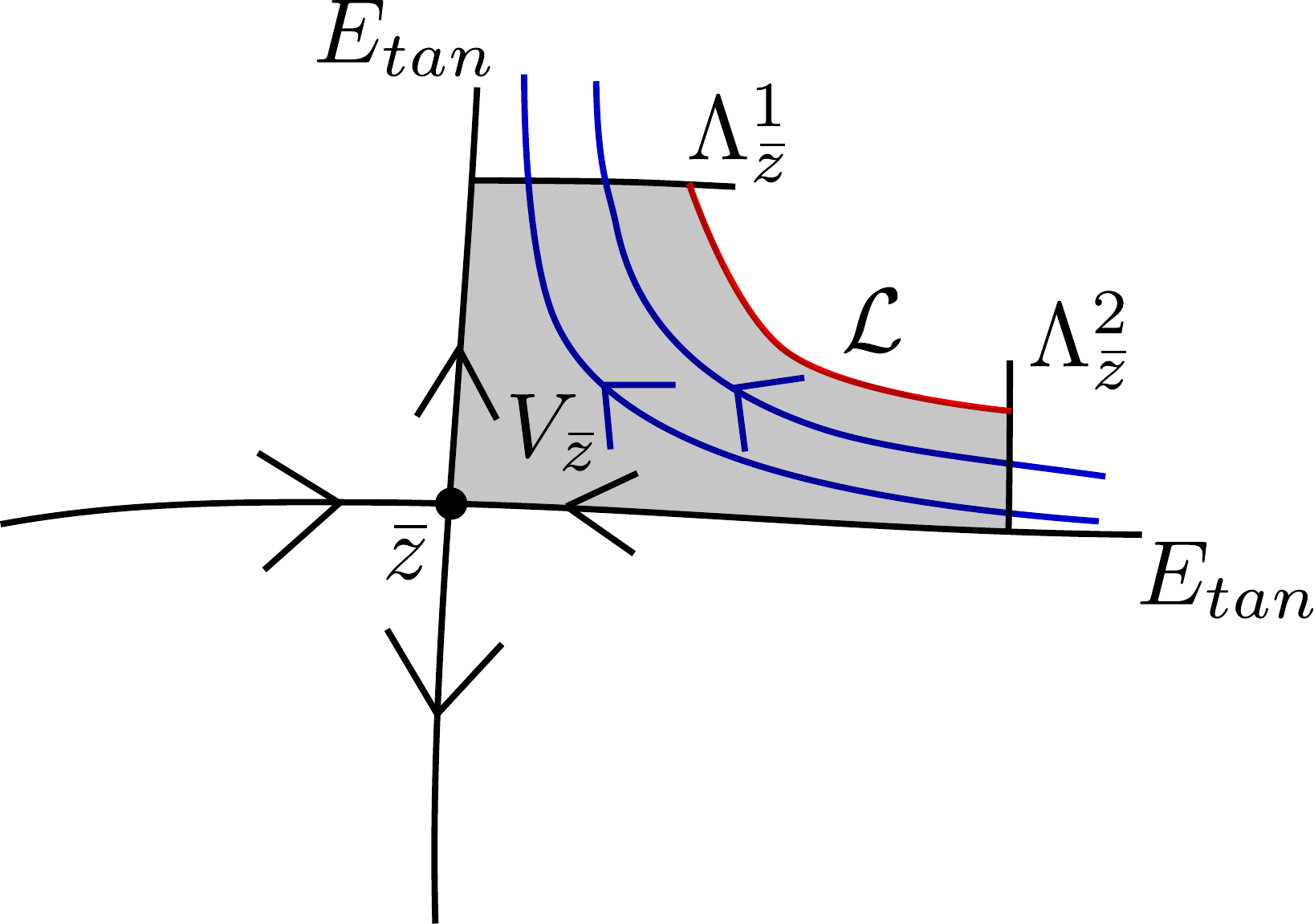}
\caption{Case (IV.ii) in Theorem \ref{thm:RSgeneral}}
\end{center}
\end{figure}
\item[(c)] There exists $\alpha,\,\beta \in \mathbb{N}^2$ such that the pulled-back metric $\pi^{\ast}(g) = g^{\ast}$ is locally bi-Lipschitz equivalent to:
\[
\begin{aligned}
h_{\bar{z}} &= (d \pmb{u}^{\alpha})^2 + (d \pmb{u}^{\beta})^2,\quad \text{ where } \quad \pmb{u}^{\alpha} = u_1^{\alpha_{1}}u_2^{\alpha_{2}} \text{ and } \, \pmb{u}^{\beta}=u_1^{\beta_{1}}u_2^{\beta_{2}}.
\end{aligned}
\]
Furthermore, there exists a vector field $\widetilde {\mathcal Z}$, which locally generates $\widetilde{\mathscr{L}}$, such that:\\
- either $|\widetilde{\mathcal{Z}}(\pmb{u}^{\alpha})| \geq  |\widetilde{\mathcal{Z}}(\pmb{u}^{\beta})|$ everywhere over $V_{\bar{z}}$, and $\widetilde{\mathcal{Z}}(\pmb{u}^{\alpha}) \neq 0$ everywhere over $V_{\bar{z}}\setminus E$;\\
- or $|\widetilde{\mathcal{Z}}(\pmb{u}^{\beta})| \geq  |\widetilde{\mathcal{Z}}(\pmb{u}^{\alpha})|$ everywhere over $V_{\bar{z}}$, and $\widetilde{\mathcal{Z}}(\pmb{u}^{\beta}) \neq 0$ everywhere over $V_{\bar{z}} \setminus E$.
\end{itemize}
\end{itemize}
\end{itemize}
\end{theorem}

\begin{proof}[Proof of Theorem \ref{thm:RSgeneral}]
Denote by $\mathscr{M}$ the reduced analytic space associated with $\Sigma$. By \cite[Theorem 1.3]{HP} (we recall the details in Theorem \ref{prop:HP} below), there exists a resolution of singularities $\pi : \mathscr{S} \to \Sigma \subset M$ via admissible blowings-up which satisfies the {Hsiang-Pati property} (see Appendix \ref{app:HP}). All blowings-up project into the singular set $Sing(\mathscr{M}) \subset S$; we can further suppose that the pre-image of $S$ is contained in the exceptional divisor (which is useful for combinatorial reasons), which guarantees (I). These properties are preserved by further real blowings-up described in Theorem \ref{prop:HP}(i)-(ii). 

Next, by \cite[Theorem 8.14]{IY} (we recall the details in Theorem \ref{thm:RSVectorField} below) we can further compose $\pi$ with a locally finite number of blowings-up of points in the exceptional divisor so that the strict transform of $\mathscr{L}$, which we denote by $\widetilde{\mathscr{L}}$, has only  elementary singularities and is either tangent or transverse to connected components of the exceptional divisor $E$. Denote by $E_{tan}$ the set of connected exceptional divisors tangent to $\mathscr{L}$, and by $E_{tr}$ the remaining ones.

Now, fix a point $\bar{z} \in \mathscr{S}$ and let $\mathcal{W}$ be a sufficiently small neighborhood of $\bar{z}$ so that $\mathcal{W}$ is orientable; in particular, fix a volume form $\omega_{\mathcal{W}}$ defined over $\mathcal{W}$. Next, up to shrinking $\mathcal{W}$, there exists a relatively compact open set $\mathcal{V} \subset M$, with $\pi(\mathcal{W}) \subset \mathcal{V}$, such that $\Delta$ is generated by a $1$-form $\delta \in \Omega^1(\mathcal{V})$ (cf. Remark \ref{rk:LocalDescriptMartinet}(i)).

Consider the vector field $\mathcal{Z}$ (which is defined over $\mathcal{W}$) given by:
\begin{equation}\label{eq:TransformFoliation}
i_{\mathcal{Z}} \omega_{\mathcal{W}} = \pi^{\ast} \delta. 
\end{equation}
By Proposition \ref{prop:DivProperty} we have that $\div_{\omega_{\mathcal{S}}}(\mathcal{Z}) \in \mathcal{Z}(\pi)$. Now, denote by $\widetilde{\mathcal{Z}}$ a local generator of $\widetilde{\mathscr{L}}$ defined over $\mathcal{W}$; we note that $\widetilde{\mathcal{Z}}$ is given by the division of $\mathcal{Z}$ by as many powers as possible of the exceptional divisor (c.f. Lemma \ref{lem:final}(ii)). It follows from Lemma \ref{lem:final} that all singularities of $\widetilde{\mathscr{L}} \cap \mathcal{W}$ are saddles, and that the foliation is tangent to connected components of the exceptional divisors where the log-rank of $\pi$ is zero (because, by equation \eqref{eq:TransformFoliation}, the vector field $\mathcal{Z}$ is divisible by powers of these exceptional divisors). In particular, the log-rank over $E_{tr}\setminus E_{tan}\cap \mathcal{W}$ must be equal to $1$. Since $\mathcal{V}$ was arbitrary, we conclude Properties (II) and (III).

Next, we provide an argument over $2$-points in order to prove $(IV)(ii)$. Let $\bar{z}$ be a point in the intersection of two connected components of $E_{tan}$. Since $\widetilde{\mathscr{L}}$ is tangent to $E_{tan}$, we deduce that $\bar{z}$ is a saddle point of $\widetilde{\mathscr{L}}$. Now, by Lemma \ref{lem:HPmetric}, the pulled back metric $g^{\ast}$ is locally (at $\bar{z}$) bi-Lipschitz equivalent to the metric 
\[
h_{\bar{z}} = (d \pmb{u}^{\alpha})^2 + (d \pmb{u}^{\beta})^2.
\] 
Recalling that $\widetilde{\mathcal{Z}}$ is a local generator of $\widetilde{\mathscr{L}}$, we 
consider the locally defined analytic set $T = \{\pmb{u}\,|\,|\widetilde{\mathcal{Z}}(\pmb{u}^{\alpha})| = |\widetilde{\mathcal{Z}}(\pmb{u}^{\beta})|\}$.

If $T$ is a 2-dimensional set then $|\widetilde{\mathcal{Z}}(\pmb{u}^{\alpha})| =  |\widetilde{\mathcal{Z}}(\pmb{u}^{\beta})|$ everywhere on a neighborhood of $\bar{z}$, and by the existence of transition maps close to saddle points (see, e.g. \cite[Section 2.4]{ADL}), we conclude easily Properties (IV.ii.a), (IV.ii.b), and (IV.ii.c).

Therefore, we can suppose that $T$ is an analytic curve. We claim that, up to performing combinatorial blowings-up (that is, blowings-up whose centers are the intersection of exceptional divisors), we can suppose that $T \subset E_{tan}$ (we recall that the argument is only for $2$-points). As a result, without loss of generality, we locally have: either $|\widetilde{\mathcal{Z}}(\pmb{u}^{\alpha})|>  | \widetilde{\mathcal{Z}}(\pmb{u}^{\beta})|$, or $|\widetilde{\mathcal{Z}}(\pmb{u}^{\beta})|>  |\widetilde{ \mathcal{Z}}(\pmb{u}^{\alpha})|$ everywhere outside the exceptional divisor $E_{tan}$. Hence, again by the existence of transition maps close to saddle points, we conclude the proof of Properties (IV.ii.a), (IV.ii.b), and (IV.ii.c).

In order to prove the claim, consider a sequence of combinatorial blowings-up so that the strict transform $T^{st}$ of $T$ does not intersect $2$-points. By direct computation over local charts, the pull-back of the metric $h_{\bar{z}}$ again satisfies equation \eqref{eq:HP12} over every $2$-point in the pre-image of $\bar{z}$ (with different $\alpha$ and $\beta$). Now, denote by $\bar{T}$ the analogue of the set $T$, but computed after the sequence of combinatorial blowings-up; since $\widetilde{\mathscr{L}}$ is a line foliation (therefore, generated by one vector field), we conclude that $T^{st}$ and $\bar{T}$ coincide everywhere outside of the exceptional divisor, which proves the claim.


Finally, let $\bar{z}$ be a point contained in only one connected component of $E_{tan}$ and assume that $\bar{z}$ is not a singularity of $\widetilde{\mathscr{L}}$. Then, up to taking a sufficiently small neighborhood of $\bar{z}$, the flow-box Theorem (see e.g. \cite[Theorem 1.12]{ADL} or \cite[Theorem 1.14]{IY}) implies properties (IV.i.a), (IV.i.b), and (IV.i.c). This concludes the proof.
\end{proof}

\begin{remark}\label{rk:tangencies}
As follows from Lemma \ref{lem:final},  $\widetilde{\mathscr{L}}$ 
is tangent to a component $F$ of the exceptional divisor if and only if so is the pull-back foliation $\pi^* \mathscr L$.  Indeed, in the language of the local vector fields generating these foliations, Lemma \ref{lem:final} yields ``$\mathcal Z$ is tangent to $F$ if and only if so is $\widetilde {\mathcal Z}$''. 
\end{remark}

\begin{lemma}[Compatibility of stratifications]\label{lem:CompatibilityStratificaiton}
Recall the notation of Lemma \ref{lem:StratFol} and of Theorem \ref{thm:RSgeneral}. Up to adding a locally finite number of points to $\Sigma^0$, we have:
\[
\pi (E_{tr}\setminus E_{tan}) \subset \Sigma_{tr}^1 , \text{ and } \pi
(E_{tan}) \subset \Sigma^0 \cup \Sigma^1_{tan}.
\]
\end{lemma}

\begin{proof}[Proof of Lemma \ref{lem:CompatibilityStratificaiton}]
We start by making two remarks:

$(1)$ up to adding a locally finite number of points to $\Sigma^0$,  without loss of generality we can assume that $\Sigma^0$ contains all points $\bar{w}\in \Sigma$ where $\pi$ has $\mbox{log rk}$ equal to $0$ over the fiber of $\bar{w}$. 

$(2)$ Let $F$ be an irreducible exceptional divisor of $E$ where the log-rank is constant equal to $1$. Then $\pi(F)$ is an analytic curve over $\Sigma$; furthermore, by expression \eqref{eq:HP1}, it follows that $d\pi|_{T F}: TF \to T\pi(F)$ is an isomorphism. In particular, $\Delta$ is tangent to $\pi(F)$ at $\bar{w} = \pi(\bar{z})$ if and only if  $\pi^* \mathscr L$ is tangent to $F$ at $\bar{z}$.
By Remark \ref{rk:tangencies}, this latter property is equivalent to the tangency of  $\widetilde{\mathscr{L}}$  to $F$ at 
$\bar{z}$.

Now, by Theorem \ref{thm:RSgeneral}(I), we know that $\pi( E_{tan} \cup E_{tr}) \subset  \Sigma^0 \cup \Sigma^1_{tan}\cup \Sigma^1_{tr}$. Therefore, by the second remark, it is clear that $\pi (E_{tr}) \subset \Sigma^1_{tr}$. Next, let $\bar{w} \in \Sigma_{tr}^1$ and note that the $\mbox{log rk}$ can be assumed to be constant equal to $1$ along the fiber $\pi^{-1}(\bar{w})$, thanks to the first remark. Moreover, if we assume by contradiction that there exists $\bar{z} \in \pi^{-1}(\bar{w})$ which belongs to $E_{tan}$, we get a contradiction with the second remark. We conclude easily.
\end{proof}

\begin{remark}
Unlike for the complex analytic spaces, a resolution map of a real analytic space 
is not necessarily surjective and its image equals the closure of the regular part.  For instance for the Whitney umbrella 
$\{(x,y,z)\in \R^3; y^2= z x^2 \}$, the singular part is the vertical line $\{x=y=0\}$, 
and the image of any resolution map equals $\{(x,y,z)\in \R^3; y^2= z x^2, z\ge 0 \}$ and 
does not contain "the handle" $\{x=y=0, z<0\}$.  
\end{remark}

\subsection{Proof of Proposition \ref{prop:TopDichSing}}\label{sec:ThmTopDichSing}

We follow the notation of Theorem \ref{thm:RSgeneral}. Without loss of generality, we may suppose that the pre-image of $\Sigma^0$ contain all points over which $\pi$ has log-rank equal to $0$. Next, we note that the singular set of $\widetilde{\mathscr{L}}$ is a locally finite set of discrete points contained in $E_{tan}$. By Lemma \ref{lem:CompatibilityStratificaiton} and the fact that $\pi$ is proper, we conclude that there exists a locally finite set of points $\widetilde{\Sigma}^0 \subset \Sigma^0 \cup \Sigma^1_{tan} = \widetilde{\Sigma}$ whose pre-image contain all singular points of $\widetilde{\mathscr{L}}$ and all points where log-rank of $\pi$ is zero. Apart from adding a locally finite number of points to $\widetilde{\Sigma}^0$, we can suppose that $\Sigma^0 \subset \widetilde{\Sigma}^0$.

Now, let $\gamma: [0,1) \to \Sigma$ be a convergent transverse-singular trajectory such that  $\bar{y}:=\lim_{t\rightarrow 1} \gamma(t) \in \widetilde{\Sigma}$. Denote by $\widetilde{\gamma}$ the strict transform of $\gamma([0,1))$ under $\pi$, that is
\[
\widetilde{\gamma} := \overline{\pi^{-1}(\gamma([0,1))\setminus \Sigma^{1}_{tr})}.
\]
By hypothesis, we know that the topological limit of $\widetilde{\gamma}$, which is defined by
\[
\omega(\widetilde{\gamma}) := \bigcap_{t\in [0,1)} \overline{\pi^{-1}(\gamma([t,1))\setminus \Sigma^{1}_{tr})},
\]
is contained in the pre-image of $\bar{y}$, say $F =\pi^{-1}(\bar{y}) \subset E_{tan}$. 

Now, suppose by contradiction that $\bar{y} \notin \widetilde{\Sigma}^0$. In this case, $\widetilde{\mathscr{L}}$ is an everywhere regular foliation over $F$, and $\pi$ has log-rank equals to $1$ over $F$. By equation \eqref{eq:HP1} and Theorem \ref{thm:RSgeneral}(IV.i.b), we conclude that the topological limit of $\widetilde{\gamma}$ must contain an open neighborhood of $F$ in $E_{tan}$, which projects into a $1$-dimensional analytic set over $\Sigma$. This is a contradiction with the definition of convergent transverse-singular trajectory, which implies that $\bar{y} \in \widetilde{\Sigma}^0$. 


We now need the following:
\begin{proposition} \label{prop:characteristic}
A convergent transverse-singular trajectory $\gamma:[0,1) \rightarrow \Sigma$ is characteristic, if and only if, the topological limit $\omega (\tilde{\gamma})$ is a singular point $\bar{z}$ of $\widetilde{\mathscr{L}}$ and in this case 
$\widetilde\gamma$ is a characteristic orbit of an analytic vector field that generates  
$\widetilde{\mathscr{L}}$ in a neighborhood of $\bar{z}$.  
\end{proposition}
\begin{proof}[Proof of Proposition \ref{prop:characteristic}]
Let $\gamma$ be a  convergent transverse-singular trajectory.  If $\omega(\widetilde{\gamma})$ contains more than one point then $\gamma$ is monodromic.  Therefore we may assume that $\omega(\widetilde{\gamma})= \{\bar{z}\}$ and then, clearly, $\bar{z}$ must be a singular point of $\widetilde{\mathscr{L}}$.  Since 
all singular points are saddles there are only finitely many orbits of the associated vector field that converge to the singular point and all of them are the characteristic orbits.  
\end{proof}

As a consequence of the last proposition, we can now prove the following result, which concludes the proof of Proposition \ref{prop:TopDichSing}.

\begin{proposition} \label{prop:characteristic-semianalytic}
Let   $\gamma:[0,1) \rightarrow \Sigma$ be a (convergent transverse-singular) characteristic trajectory, then $\gamma([0,1))$ is a semianalytic curve. \end{proposition}
\begin{proof}[Proof of Proposition \ref{prop:characteristic-semianalytic}]
The strict transform $\widetilde \gamma$ of $\gamma$ is a characteristic orbit of a saddle singularity, and therefore, it 
is semianalytic by the stable manifold Theorem of Briot and Bouquet \cite{bb1856} (see, e.g. \cite[Theorem 2.7]{ADL}).  To conclude, we note that the image of a semianalytic curve by a proper analytic map is semianalytic, see Remark \ref{rk:semiref}.
\end{proof}



\subsection{Proof of Proposition \ref{prop:ComparisonMonodromic}}\label{sec:ProofPropComparisonMonodromic}

We follow the notation of Theorem \ref{thm:RSgeneral}, Proposition \ref{prop:ComparisonMonodromic}, and Remark \ref{rk:ProofPropComparisonMonodromic}. Without loss of generality, we may assume that there exists an open neighborhood $\mathcal{W}$ of $\bar{y}$ such that $\gamma([0,1)) \subset \mathcal{W}$, $\lambda([0,1)) \subset \mathcal{W}$, and either $\gamma ([0,1))\cap \Sigma_{tr}^1 = \lambda ([0,1))\cap \Sigma_{tr}^1 = \emptyset$ with sequences of times  $\{t_k^{\gamma}\}_{k\in \N}, \{t_k^{\lambda}\}_{k\in \N}$, or $\gamma ([0,1))\cap \Sigma_{tr}^1 \neq \emptyset$, $\lambda ([0,1))\cap \Sigma_{tr}^1 \neq \emptyset$ are infinite with a common sequence of jumps $\{(i_k,j_k)\}_{k\in \N}$ associated respectively with $\{t_k^{\gamma}\}_{k\in \N}$ and $\{t_k^{\lambda}\}_{k\in \N}$.

The Riemmanian metric $g$ is bi-lipschitz equivalent to an analytic metric over $\mathcal{W}$. Since Proposition \ref{prop:ComparisonMonodromic} is invariant by local bi-lipschitz equivalence of metrics, we suppose without loss of generality that $g$ is analytic. We denote by $\mathcal{V}$ the pre-image of $\mathcal{W}$ by $\pi$, and by $g^{\ast}$ the pull-back of $g$ by $\pi$ (which is analytic and degenerated over $E$).

We recall that $\gamma(0)$ and $\lambda(0)$ are assumed to belong to the same section $\Lambda$ and that $d^{\Lambda}(\lambda(0))<d^{\Lambda}(\gamma(0))$. We denote by $\widetilde{\gamma}$ and $\widetilde{\lambda}$ the strict transform of $\gamma$ and $\lambda$ (defined  as in the proof of Proposition \ref{prop:TopDichSing}).

Since the transition maps $T^{\Lambda}$ and $T^{j_k}_{i_k}$ satisfy  property  \eqref{eq:PoincareTransitionProp} and \eqref{eq:TransitionPropI} respectively, we note that
\begin{eqnarray}\label{eq:distance-section}
d^{\Lambda}\left(\gamma(t^{\gamma}_k)\right) < d^{\Lambda} \left(\lambda(t^{\lambda}_k)\right) \qquad \forall\,k\in \N
\end{eqnarray}
in the case $\gamma ([0,1))\cap \Sigma_{tr}^1 = \lambda ([0,1))\cap \Sigma_{tr}^1 = \emptyset$, and
\begin{eqnarray}\label{eq:distance-section2}
d^{i_k}\left(\gamma(t_k^{\gamma})\right) < d^{i_k}\left(\lambda(t_k^{\lambda})\right) \qquad \forall  k\in \N
\end{eqnarray}
in the case where $\gamma ([0,1))\cap \Sigma_{tr}^1 \neq \emptyset$ and $\lambda ([0,1))\cap \Sigma_{tr}^1 \neq \emptyset$ are infinite.

Finally, since $\pi$ is a proper morphism, in order to prove Proposition \ref{prop:ComparisonMonodromic} it is enough to show a similar result, locally, at every point on the resolution space which belongs to the topological limit $\omega(\widetilde{\gamma}) = \omega(\widetilde{\lambda})$.

Since $\gamma$ is monodromic, if a point $\bar{z} \in \omega(\widetilde{\gamma})$ is a saddle of $\widetilde{\mathscr{L}}$, then there are two connected components of $E_{tan}$ which contain $\bar{z}$ (in other words, it satisfies the normal form given in Theorem \ref{thm:RSgeneral}(IV.ii)). Therefore, for each $\bar{z} \in \omega(\widetilde{\gamma})$, either the normal form (IV.i) or (IV.ii) of Theorem \ref{thm:RSgeneral} is verified. We study these two possibilities separately (we do not distinguish $(IV.i.b)$ and $(IV.i.c)$ in this part of the proof). In both cases, given local sections $\Lambda_{\bar{z}}^{j}$ for $j=1,2$  we consider the distance functions
\[
d_{\bar{z}}^j: \Lambda_{\bar{z}}^{j} \to \mathbb{R}
\]
where $d_{\bar{z}}^j(p)$ is the length (via $g^{\ast}$) of the half arc contained in $\Lambda_{\bar{z}}^j$ that joins $p$ to $E_{tan}$, and $\mbox{length}^{g^{\ast}}$ denotes the length with respect to $g^{\ast}$.

The next lemma handles the first case.

\begin{lemma}\label{lem:lengthRegular}
Recalling the notation of Theorem \ref{thm:RSgeneral}(IV.i),
assume that  there exists only one connected component of $E_{tan}$ which contains $\bar z$. For each point $p \in \Lambda^1_{\bar{z}}$, denote by $L(p)$ the half-leaf of $\widetilde{\mathscr{L}}$ whose boundary is given by $p$ and $\phi_{\bar{z}}(p)$. Then there exists $\epsilon>0$ (which depends only on the neighborhood of $\bar{z}$) such that
\[
d^1_{\bar{z}}(p) < d^1_{\bar{z}}(q)< \epsilon \quad \implies \quad {\rm length}^{g^{\ast}}(L( p)) \leq {\rm length}^{g^{\ast}}(L(q))
\]
\end{lemma}
\begin{proof}[Proof of Lemma \ref{lem:lengthRegular}]
Note that $\widetilde{\mathscr{L}}$ is non-singular, so there exists a non-singular locally defined vector field $\widetilde{\mathcal{Z}}$ which generates $\widetilde{\mathscr{L}}$. Denote by $\varphi^{\widetilde{\mathcal{Z}}}_p(t)$ the flow of $\widetilde{\mathcal{Z}}$ with time $t$ and initial condition $p$. Since $\widetilde{\mathcal{Z}}$ is non-zero, for each $p\in \Lambda^1_{\bar{z}}$ the minimal time $t_p$ so that $\varphi_p^{\widetilde{\mathcal{Z}}}(t_p) \in \Lambda^2_{\bar{z}}$ is an analytic function in $p$. It follows that the function
\[
{\rm length}^{g^{\ast}}(L(p)) = \int_{0}^{t_p} \left| {\widetilde{\mathcal{Z}}}(\varphi^{\widetilde{\mathcal{Z}}}_p(s))\right|^{g^{\ast}}ds
\]
is analytic over $\Lambda^1_{\bar{z}}$, since all objects are analytic. Furthermore, ${\rm length}^{g^{\ast}}(L(p)) \geq 0$, and it is equal to zero if and only if $p\in E$. This implies the desired monotonicity property.
\end{proof}

We now handle the second case.

\begin{lemma}\label{lem:lengthSaddle}
Recalling the notation of Theorem \ref{thm:RSgeneral}(IV.ii),
assume that there exist two connected component of $E_{tan}$ which contains $\bar z$. For each point $p \in \Lambda^1_{\bar{z}}$, denote by $L(p)$ the half-leaf of $\widetilde{\mathscr{L}}$ whose boundary is $p$ and $\phi_{\bar{z}}(p)$. Then there exists $K>0$ and $\epsilon>0$ (which depends only on the neighborhood of $\bar{z}$) such that
\[
d^1_{\bar{z}}(p) < d^1_{\bar{z}}(q)<\epsilon \quad \implies \quad {\rm length}^{g^{\ast}}(L( p)) \leq K\, {\rm length}^{g^{\ast}}(L(q))
\]
\end{lemma}
\begin{proof}[Proof of Lemma \ref{lem:lengthSaddle}]
Up to changing $K>0$, it suffices to prove the result for a metric that is  bi-lipschitz equivalent to $g^{\ast}$. Therefore, without loss of generality we assume that $g^{\ast} = h_{\bar{z}} = (d \pmb{u}^{\alpha})^2+(d \pmb{u}^{\beta})^2$ (see Theorem \ref{thm:RSgeneral}(IV.ii.b)). Although the definition of $\alpha$ and $\beta$ are not symmetric (see \eqref{eq:HP12}), this does not interfere in this part of the proof; so we assume that $\widetilde{\mathcal{Z}}(\pmb{u}^{\alpha}) \neq 0$ everywhere outside of $E$, and that $|\widetilde{\mathcal{Z}}(\pmb{u}^{\alpha})|\geq |\widetilde{\mathcal{Z}}(\pmb{u}^{\beta})|$ everywhere in the neighborhood of $\bar{z}$, where $\widetilde{\mathcal{Z}}$ is a local generator of $\widetilde{\mathscr{L}}$ (the other case is analogous).

Denote by $\varphi^{\widetilde{\mathcal{Z}}}_p(t)$ the flow of $\widetilde{\mathcal{Z}}$ with time $t$ and initial condition $p$. For each $p\in \Lambda^1_{\bar{z}}\setminus E$ the minimal time $t_p$ so that $\varphi_p^{\widetilde{\mathcal{Z}}}(t_p) \in \Lambda^2_{\bar{z}}$ is an analytic function over $\Lambda^1_{\bar{z}}\setminus E$, but it does not admit an analytic extension to $\Lambda^1_{\bar{z}}\cap E$. In particular, the function
\[
{\rm length}^{g^{\ast}}(L(p)) = \int_{0}^{t_p} \left| \widetilde{\mathcal{Z}}(\varphi_p^{\widetilde{\mathcal{Z}}}(s))\right|^{g^{\ast}}ds
\]
does not admit an analytic extension to $\Lambda^1_{\bar{z}}\cap E$. Nevertheless, we note that:
\[
\begin{aligned}
{\rm length}^{g^{\ast}}(L(p)) &= \int_{0}^{t_p} \sqrt{\left[\widetilde{\mathcal{Z}}(\pmb{u}^{\alpha})\circ\varphi_p^{\widetilde{\mathcal{Z}}}(s)\right]^2  + \left[\widetilde{\mathcal{Z}}(\pmb{u}^{\beta})\circ\varphi_p^{\widetilde{\mathcal{Z}}}(s)\right]^2}\,ds\\
&=\int_{0}^{t_p} \sqrt{ d(\pmb{u}^{\alpha} \circ \varphi_p^{\widetilde{\mathcal{Z}}}(s) )^2 + d(\pmb{u}^{\beta} \circ \varphi_p^{\widetilde{\mathcal{Z}}}(s) )^2}\, ds
\end{aligned}
\]
Now, since $\widetilde{\mathcal{Z}}(\pmb{u}^{\alpha})$ is analytic and vanishes only on $E$, we conclude that $d (\pmb{u}^{\alpha} \circ \varphi_p^{\widetilde{\mathcal{Z}}}(s)) $ is of constant sign outside of $E$. On the one hand, this implies that
\begin{equation}\label{eq:metricInequalities1}
\begin{aligned}
{\rm length}^{g^{\ast}}(L(p)) &\geq \int_{0}^{t_p} |d(\pmb{u}^{\alpha} \circ \varphi_p^{\widetilde{\mathcal{Z}}}(s) |ds\\
&= \left|\pmb{u}^{\alpha}(p) - \pmb{u}^{\alpha}\circ \phi_{\bar{z}}(p)\right|
\end{aligned}
\end{equation}
On the other hand, from the fact that $|d(\pmb{u}^{\alpha} \circ \varphi_p^{\widetilde{\mathcal{Z}}}(s)) | \geq |d(\pmb{u}^{\beta} \circ \varphi_p^{\widetilde{\mathcal{Z}}}(s)) |$, we conclude that
\begin{equation}\label{eq:metricInequalities2}
\begin{aligned}
{\rm length}^{g^{\ast}}(L(p)) &\leq \sqrt{2} \int_{0}^{t_p} |d(\pmb{u}^{\alpha} \circ \varphi_p^{\widetilde{\mathcal{Z}}}(s) |ds\\
&= \sqrt{2} \left|\pmb{u}^{\alpha}(p) - \pmb{u}^{\alpha}\circ \phi_{\bar{z}}(p)\right|
\end{aligned}
\end{equation}
Although the function $\phi_{\bar{z}}(p)$ is analytic outside of $E$, it does not admit an analytic extension to $E$ and the treatment of this case differs from the one in Lemma \ref{lem:lengthRegular}.

In order to be precise, denote by $\lambda_i:[0,1] \to \Lambda^i_{\bar{z}}$ an analytic parametrizations of the sections $\Lambda^i_{\bar{z}}$ such that $\lambda_i(0) \in E$, 
for $i=1,2$. They can be always choosen so that $\pmb{u}^{\alpha} \circ \lambda_i(t) = t^{a_i}$ for some $a_i \in \mathbb{N}$. Now, by \cite[Theorem 1.4]{speiss2}, the composition $\lambda_2^{-1} \circ \phi_{\bar{z}} \circ \lambda_1 \circ (\exp (-x))$ belongs to a Hardy field $\mathcal{F}$ of germs of function at infinity which also contains the exponential function. Thus, since it is a field, it follows that also the function
\[
\begin{aligned}
\pmb{u}^{\alpha} \circ \lambda_1 \circ (\exp (-x))- \pmb{u}^{\alpha} \circ \phi_{\bar{z}}\circ \lambda_1 \circ (\exp (-x)) =\\
= \left(\exp (-x)\right)^{a_1} -  \left(\lambda_2^{-1} \circ\phi_{\bar{z}}\circ \lambda_1 \circ (\exp (-x))\right)^{a_2}
\end{aligned}
\]
belongs to $\mathcal{F}$. Since $\mathcal{F}$ is a Hardy field, Lemma \ref{lem:HardyField} implies that the function
\[
\pmb{u}^{\alpha}\circ \lambda_1(t) - \pmb{u}^{\alpha}\circ \phi_{\bar{z}} \circ \lambda_1(t)
\]
is monotone for $t$ sufficiently close to $0$ (that is, for $p = \lambda_1(t)$ sufficiently close to $E$). We conclude easily from this observation and the inequalities \eqref{eq:metricInequalities1} and \eqref{eq:metricInequalities2}.

\end{proof}

As observed before, Proposition \ref{prop:ComparisonMonodromic} now follows from the two lemmas above and the previous considerations made in all this section.


\appendix

\section{Singular horizontal paths}\label{SECSingular}

Let $M$ be a smooth connected manifold of dimension $n\geq 3$ and $\Delta$ a totally nonholonomic distribution of rank $m<n$ on $M$. As in the introduction, we assume that $\Delta$ is generated globally on $M$ by a family of $k$ smooth vector fields $X^1, \ldots, X^k$. For every $i=1, \ldots, k$,  we define the Hamiltonian $h^i :T^*M \rightarrow \R$ by 
$$
h^i(x,p) := p\cdot X^i(x) \qquad \forall\,\psi=(x,p) \in T^*M,
$$
and denote by $\vec{h}^i$ the corresponding Hamiltonian vector field on $T^*M$ with respect to canonical symplectic structure $\omega$.  Then we set
$$
\vec{\Delta}_\psi :=  \Bigl\{\vec{h}^1(\psi), \ldots, \vec{h}^k(\psi)\Bigr\} \qquad \forall\, \psi \in T^*M.
$$
By construction, $\vec{\Delta}$ is a smooth distribution of rank $m$ in $T^*M$ which does not depend on the choice of the family $X^1, \ldots, X^k$. Let $\Delta^{\perp}$ be  the annihilator of $\Delta$ in $T^*M$, defined by  
$$
\Delta^\perp:=\{(x,p)\in T^*M\,\vert\,p\perp \Delta_x\}\subset T^*M,
$$
and $\pi:T^*M\to M$ the canonical projection. Singular horizontal paths  can be characterized as follows (we leave the reader to check that  \cite[Proposition 1.11]{riffordbook} can be formulated in this way).  

\begin{proposition}\label{PROPsing}
Let $\gamma : [0,1] \rightarrow M$  be fixed, then the two following properties  are equivalent:
\begin{itemize}
\item[(i)] $\gamma$ is a singular horizontal path with respect to $\Delta$.
\item[(ii)] There exists  $\psi:[0,1]\rightarrow \Delta^{\perp}\setminus \{0\}$ such that $\psi$ is
a horizontal path with respect to $\vec{\Delta}$ and $\pi(\psi)=\gamma$.
\end{itemize}
\end{proposition}

The following characterization is due to Hsu \cite{hsu91} and plays a major role in the proof of Proposition \ref{prop:MonoFiniteLenght}. We recall that, for every $\psi \in T^*M$, $\mbox{ker} (\omega_{\vert \Delta^{\perp}})_{\psi}$ is defined as
$$
\mbox{ker} (\omega_{\vert \Delta^{\perp}})_{\psi} := \Bigl\{ \xi \in T_{\psi}\Delta^{\perp} \, \vert \, \omega_{\psi}(\xi,\eta)=0, \, \forall\,\eta \in T_{\psi}\Delta^{\perp} \Bigr\} =  \left( T_{\psi} \Delta^{\perp} \right)^{\omega} \cap \, T_{\psi} \Delta^{\perp},
$$
where $ \left( T_{\psi} \Delta^{\perp} \right)^{\omega} $    denotes the symplectic complement of $ T_{\psi} \Delta^{\perp}$. 

\begin{proposition}\label{PROPcharacteristic}
Let $\gamma : [0,1] \rightarrow M$  be fixed, then the two following properties  are equivalent:
\begin{itemize}
\item[(i)] $\gamma$ is a singular horizontal path with respect to $\Delta$.
\item[(ii)] There exists an absolutely continuous curve $\psi:[0,1]\rightarrow \Delta^{\perp}\setminus\{0\}$ with derivatives in $L^2$ such that $\pi(\psi)=\gamma$ and 
$$
\dot{\psi}(t)\in \mbox{ker} (\omega_{\vert \Delta^{\perp}})_{\psi(t)} \quad \mbox{for a.e. } t\in [0,1].
$$
\end{itemize}
\end{proposition}

\begin{proof}[Proof of Proposition \ref{PROPcharacteristic}]
First, we  note that  
$$
T_{\psi} \Delta^{\perp} = \bigl(\vec{\Delta} _\psi\bigr)^{\omega} \qquad \forall \,\psi \in \Delta^{\perp}.
$$
As a matter of fact, if $\xi \in T_{\psi}(T^*M)$ satisfies $\xi\cdot h^i (\psi)=0$ for some $i=1, \ldots,m,$ then $\omega_{\psi}(\xi,\vec{h}^i)=0$. This shows that $T_{\psi} \Delta^{\perp}$ is contained in the symplectic complement of $\vec{\Delta}_\psi $. Thus, since both spaces have the same dimension $2n-m$, they must coincide.

Thanks to this fact we deduce that $\left( T_{\psi} \Delta^{\perp} \right)^{\omega}=\vec{\Delta}_\psi$, hence 
 \begin{eqnarray}\label{eqcharac}
\mbox{ker} (\omega_{\vert \Delta^{\perp}})_{\psi} = \left( T_{\psi} \Delta^{\perp} \right)^{\omega} \cap \, T_{\psi} \Delta^{\perp} = \vec{\Delta}_\psi  \cap \, T_{\psi} \Delta^{\perp} \qquad \forall \,\psi \in \Delta^{\perp}.
\end{eqnarray}

Let us now prove that (i) $\Rightarrow$ (ii). By Proposition \ref{PROPsing}, if $\gamma : [0,1] \rightarrow M$  is a singular horizontal path with respect to $\Delta$, it is the projection of a curve $\psi:[0,1]\rightarrow \Delta^{\perp}$ which is horizontal with respect to $\vec{\Delta}$. Thus $\dot{\psi}(t)\in \vec{\Delta}_{\psi(t)}\cap T_{\psi(t)} \Delta^{\perp}$ for a.e. $t \in [0,1]$, and \eqref{eqcharac} yields the result.

To prove that (ii) $\Rightarrow$ (i) it suffices to observe that, thanks to (\ref{eqcharac}) and Proposition \ref{PROPsing}, any absolutely continuous curve $\psi:[0,1]\rightarrow \Delta^{\perp}$ with derivatives in $L^2$ and with $\dot{\psi}(t)\in \mbox{ker} (\omega_{\vert \Delta^{\perp}})_{\psi(t)}$ projects onto a curve which is horizontal and singular with respect to $\Delta$.
\end{proof}

Finally, we conclude this section with a uniform bound on the norm of the lift of singular horizontal paths. For this purpose, we equip $M$ with a Riemannian metric $g$, and denote the corresponding dual norm on $T^*M$ as $|\cdot |^*$ (namely, $|\psi|^*$ stands for the norm of $p$ at $x$, where $\psi=(x,p)$ and $p\in T_x^*M$).

\begin{proposition}\label{PROPpbounded}
Let $\mathcal{K}$ be a compact set in $M$ and $\ell>0$ be fixed. Then there is $K=K(\mathcal{K},\ell)>0$ such that, given a singular horizontal path $\gamma :[0,1] \rightarrow \mathcal{K}$ of length less than $\ell$, any lift $\psi:[0,1]\rightarrow \Delta^{\perp}$ given by Proposition \ref{PROPsing}(ii) or \ref{PROPcharacteristic}(ii) with $|\psi(0)|^*\leq 1$ satisfies 
$$
|\psi(t)|^* \leq K \qquad \forall \,t\in [0,1].
$$
\end{proposition}

\begin{proof}[Proof of Proposition \ref{PROPpbounded}]
Let $\mathcal{K}$ be a compact set in $M$ and $\ell>0$ be fixed. For each $x\in \mathcal{K}$, there is an open neighborhood $U^x$ of $x$ and $m$ smooth vector fields $X^1_m, \ldots, X^m_x$ defined over a neighborhood $V^x$ of $\bar{U}^x$ such that 
$$
\Delta_y = \mbox{Span} \Bigl\{ X^1_x(y), \ldots, X^m_x(y)\Bigr\} \qquad \forall\,y \in U^x.
$$
By compactness of $\mathcal{K}$ there are $x_1, \ldots, x_N$ in $\mathcal{K}$ such that 
$$
\mathcal{K}\subset \bigcup_{i=1}^N U^{x_i}.
$$
Therefore if we multiply each $X_{x_i}$ by a cut-off function vanishing outside $V^{x_i}$, we can construct a family of smooth vector fields $X^1, \ldots, X^k$ (with $k=m\cdot N$) on $M$ which generate $\Delta$ over $\mathcal{K}$ such that for every $x\in \mathcal{K}$ and every $v\in \Delta_x \subset T_xM$ there is $u\in \R^k$ such that 
$$
v= \sum_{i=1}^k u_i X^i (x) \qquad \mbox{and} \qquad C_1 \, |v| \leq |u| \leq C_2 \, |v|,
$$
for some constants $C_1, C_2>0$ independent of $v$.  Let $\gamma :[0,1] \rightarrow \mathcal{K}$  be a singular horizontal path  of length less than $\ell$ and $\psi=(\gamma,p):[0,1]\rightarrow \Delta^{\perp}$  a lift of $\gamma$ (given by Proposition \ref{PROPsing}(ii) or \ref{PROPcharacteristic}(ii)) with $|\psi(0)|^*\leq 1$. There is $u\in L^2([0,1],\R^k)$ such that for a.e. $t\in [0,1]$, there holds
$$
\left\{
\begin{array}{rcl}
\dot{\gamma} (t) & = & \sum_{i=1}^k u_i(t) X^i(\gamma(t)) \\
\dot{p}(t) & = & - \sum_{i=1}^k u_i(t) \, p(t)\cdot d_{\gamma(t)} X^i
\end{array}
\right.
$$
and
$$
C_1 \, |\dot{\gamma}(t)| \leq |u(t)| \leq C_2 \, |\dot{\gamma}(t)|.
$$
Then, if we define $\alpha:[0,1] \rightarrow \R$ by $\alpha(t) := 1+|\psi(t)|^*$, there holds
$$
\dot{\alpha}(t) \leq C |u(t)| \, \alpha(t) \leq C |\dot{\gamma}(t)| \, \alpha(t).
$$
Thus, noticing that $\alpha(0)\leq 2$, it follows by Gronwall Lemma that 
$$
\alpha(t)\leq \alpha(0)\,e^{C\int_0^t|\dot\gamma|}\leq 2\,e^{C\ell}, 
$$
as desired.
\end{proof}

\section{Semianalytic curves}\label{App:semianalytic}

We recall the basic facts on semianalytic sets of dimension 1 that we need in this paper. For detailed presentations of semianalytic sets we refer the reader to  \cite{lojasiewicz}, \cite{BMihes}. 

Let $M$ be a real analytic manifold of dimension $n$. A subset $X$ of $M$ is semianalytic if each $y\in M$ has a neighborhood $U$ such that $X\cap U$
 is a finite union of sets of the form 
 $$
 \Bigl\{x\in U \, \vert \, f(x) = 0,  g_1 (x)  < 0, \ldots ,g_l(x)< 0\Bigr\},
 $$ 
 with $f,g_1,\ldots ,g_l :\to \R$ analytic.  Every semianalytic set admits a locally finite stratification into nonsingular strata which are locally closed analytic submanifolds of dimension $k \in \{0,\ldots,n\}$ and semianalytic sets. The dimension of a semianalytic set is defined as the maximum of the dimensions of its strata, and it coincides with its Hausdorff dimension.  

 \begin{definition} 
In this paper we call \emph{a semianalytic curve of $M$} any compact connected semianalytic subset of $M$ of Hausdorff dimension at most $1$.  
 \end{definition}

\begin{remark}\label{rk:semiref}
The image of a semianalytic curve by an  analytic map 
is again semianalytic. This follows for instance from \cite[Theorem 1, p. 92]{lojasiewicz}.  
(We note that this property fails for compact semianalytic sets of higher dimension.  
In this case  such images are not necessarily semianalytic, but they are subanalytic, cf. \cite{BMihes}.) 
\end{remark}

By the existence of a Whitney regular stratification, see \cite{lojasiewicz} or \cite{wall}, 
we have the following.:

\begin{lemma}[Regular stratification of $X$]\label{lem:Stratcurve}
Let $X$ be a semianalytic curve of $M$. Then there exists a stratification of $X$,
$$
X = X^0 \cup X^1,
$$
that satisfies the following properties:  $X^0$ is finite and $X^1$ is a finite union of 1-dimensional strata.  
Every 1-dimensional stratum $\Pi$ is  a connected locally closed analytic submanifold of $M$ and a semianalytic set, and, moreover, its closure $\overline \Pi$
in $M$ is a $C^ 1$ submanifold with boundary. 
\end{lemma}

Similarly to Remark \ref{rk:LocalTrivialityPuiseux}, we have the following strengthening of the above property.

\begin{remark}\label{rk:LocalPuiseux}\hfill
\begin{itemize}
\item[(i)]
By \cite[Proposition 2]{pawlucki} or  \cite{wall} (proof of Proposition p.342), any stratification as in Lemma \ref{lem:Stratcurve} satisfies, in addition, the following property:   For every 1-dimensional stratum $\Pi$ and every  $p\in \overline \Pi \setminus \Pi$ there exist a positive integer $k$ and  a local system of analytic coordinates 
$(x_1,x')$, $x'=(x_2, \ldots , x_{n})$, at $p$ such that $\overline \Pi$ is given by the graph $\{x'= \varphi (x_1)\}$, defined locally on $\{x_1\ge 0 \}$, where $\varphi $ is $C^1$ and the mapping $t\mapsto \varphi  (t^k)$ is analytic.  This implies in particular that  $\varphi$ is of class $C^{1,1/k}$ (see Remark \ref{rk:LocalTrivialityPuiseux} for a proof of the latter property). 
\item[(ii)] 
It follows by the above results that every semianalytic curve admits a continuous piecewise analytic parameterization $\gamma : [0,1] \to M$. In other words, there exists a finite set 
$0=t_0 < t_1 \cdots <t_N =1$ such that $\gamma$ restricted to each subinterval 
$[t_i, t_{i+1}]$ is analytic (i.e. extends analytically through the endpoints), and the endpoints are the only possible critical points of such restriction.  
\end{itemize}
\end{remark}





\section{Hardy fields}\label{APPHardy}

\emph{A Hardy field} $\mathcal{F}$ is a field of germs at $+\infty $ of functions from $\R$ to $\R$ (that is, for each $f\in \mathcal{F}$, there exists $\bar{z}_f \in \mathbb{R}$ such that $f:(\bar{z}_f,\infty)\to \mathbb{R}$ is well-defined) which is closed under differentiation. Since every non-zero function in $\mathcal{F}$ admits an inverse (in $\mathcal{F}$), any element of a Hardy field is eventually either strictly positive, strictly negative, or zero. Therefore, since the field is closed under differentiation, it is well-known that:

\begin{lemma}\label{lem:HardyField}
Let $f:(\bar{z},\infty)\to \mathbb{R}$ be a function in a Hardy field. Then there exists $\bar{w} \in (\bar{z},\infty)$, such that the restriction $f|_{(\bar{w},\infty)}:(\bar{w},\infty) \to \mathbb{R}$ is either strictly monotone or constant.
\end{lemma}


Following a work of Ilyashenko,  Speissegger constructs in \cite{speiss2} a Hardy field $\mathcal F$ that contains, via composition with $-\log t$, all transition maps of hyperbolic singularities (i.e. saddles) of planar analytic vector fields, or equivalently all such transition maps composed with $\exp (-x)$ belong to $\mathcal F$.  In particular, it follows that all algebraic combinations, sums and products, of such transition maps are monotone.  We use this result in the proof of Lemma \ref{lem:lengthSaddle}.

\section{Resolution of singularities}\label{app:ResolutionOfSingularities}

In what follows, $M$ is a real-analytic manifold and we denote by $\mathcal{O}_M$ the sheaf of analytic functions over $M$. Given a point  $\bar{z} \in M$, we denote by $\mathcal{O}_{M,\bar{z}}$ the analytic function germs at $\bar{z}$ and by $m_{\bar{z}}$ the maximal ideal of $\mathcal{O}_{M,\bar{z}}$. Given an ideal sheaf $\mathcal{I}$ of $\mathcal{O}_M$ and a point $\bar{z}\in M$, the order of $\mathcal{I}$ at $\bar{z}$ is defined by
\[
\sup\{r\in \mathbb{N}\,\vert \, \mathcal{I} \cdot \mathcal{O}_{M,\bar{z}} \subset m_{\bar{z}}^r \}
\]
The \emph{zero set} of $\mathcal{I}$, which we denote by $V(\mathcal{I})$, is the set of points where the order of $\mathcal{I}$ is at least one.

\subsection{Blowings-up}\label{SSEC:SNC}


\paragraph{Blowing-up of $\mathbb{R}^n$.} We start by briefly recalling the definition of blowings-up over $\mathbb{R}^n$ (see \cite[Sections 8B and 8C]{IY} or \cite[Section 3.1]{ADL} for a more complete discussion). Let us fix a coordinate system $(x_1, \ldots, x_n)$ of $\mathbb{R}^n$ and we consider a sub-manifold $\mathcal{C} = \{ x_1 =0,\ldots, x_t=0\}$ for some $1\leq t\leq n$. Let $\mathbb{P}^{t-1}$ be the real projective space of dimension $t-1$ with homogenous coordinates $(y_1,\ldots, y_t)$. We consider the set $\widetilde{M} \subset \mathbb{R}^{n} \times \mathbb{P}^{t-1}$ given by:
\[
\widetilde{M} = \left\{(x,y) \in \mathbb{R}^{n} \times \mathbb{P}^{t-1}\,\vert \, x_i y_j = x_j y_i, \, \forall\,i=1,\ldots,n,\,j=1,\ldots,t \right\}.
\]
Note that $\widetilde{M}$ is an analytic manifold. The restriction of the projection map $\tau: \mathbb{R}^{n} \times \mathbb{P}^{t-1} \to \mathbb{R}^{n}$ to $\widetilde{M}$, which we denote by $\sigma:\widetilde{M} \to \mathbb{R}^{n}$, is called a \emph{blowing-up} (with center $\mathcal{C} = \{ x_1 =0,\ldots, x_t=0\}$). The set $F = \sigma^{-1}(\mathcal{C})$ is said to be the \emph{exceptional divisor} of $\sigma$. Note that $\sigma$ is a diffeomorphism from $\widetilde{M} \setminus F$ into its image $\mathbb{R}^n \setminus \mathcal{C}$.

\paragraph{Blowings-up in general manifolds.} It is well-known that the definition of blowing-up can be extended to general analytic manifolds. This means that, given a nonsingular analytic irreducible submanifold $\mathcal{C}$ of $M$, there exists a proper analytic map $\sigma: \widetilde{M} \to M$ such that, at every point $\bar{z} \in \mathcal{C}$, the map $\sigma$ locally coincide with the model of the previous paragraph. The sub-manifold $\mathcal{C}$ is said to be the \emph{center} of the blowing-up $\sigma$. For a precise definition and further details about blowings-up, we refer the reader to \cite[Section II.7]{hart}.

\paragraph{Simple normal crossing divisors.} A smooth divisor $F$ over $M$ is a nonsingular and connected
analytic hypersurface of $M$. We denote by $\mathcal{I}_F$ the reduced and coherent ideal sheaf of $\mathcal{O}_M$ whose zero set is $F$. Note that, at each point $\bar{z} \in F$, there exists a coordinate system $x= (x_1,\ldots,x_n)$ centered at $\bar{z}$ such that, locally, $\mathcal{I}_F = (x_1)$. 

A \emph{simple normal crossing divisor} over $M$, which we call \emph{SNC divisor} for short, is a set $E$ which is a union of smooth divisors and, at each point $\bar{z} \in E$, there exists a coordinate system $x= (x_1,\ldots,x_n)$ centered at $\bar{z}$ such that, locally, $E = (x_1 \cdots x_l=0)$ for some $1 \leq l\leq n$. 

\begin{remark}
In the literature, a SNC divisor $E$ is a finite union of divisors. Here, we allow $E$ to have countable number of divisors in order to simplify the notation used for a resolution of singularities in the analytic category (c.f. Definition \ref{def:SeqAdmBlowings-up} below).  Indeed, in the analytic category it is usual to present resolution of singularities in term of relatively compact sets $U\subset M$; note that the restriction $E \cap U$ is a finite union of divisors.
\end{remark}

Given a blowing-up $\sigma: \widetilde{M} \to M$ with center $\mathcal{C}$, we note that the pre-image of $\mathcal{C}$ is a divisor, which we call the \emph{exceptional divisor} associated to $\sigma$.

\paragraph{Admissible blowings-up.} Consider a SNC divisor $E$ over $M$. A blowing-up $\sigma: \widetilde{M} \to M$ is said to be \emph{admissible} (in respect to $E$) if the center of blowing-up $\mathcal{C}$ has normal crossings with $E$, that is, at each point $\bar{z} \in \mathcal{C}$, there exists a coordinate system $x=(x_1,\ldots,x_n)$ centered at $\bar{z}$ and a sub-list $(i_1,\dots, i_t)$ of $(1,\ldots, n)$ such that, locally, $E = (x_1 \cdots x_l =0)$ and $\mathcal{C} = (x_{i_1}=0,\ldots,x_{i_t}=0)$. Now, consider the set $\widetilde{E}$ given by the union of the pre-image (under $\sigma$) of $E$ with the exceptional divisor $F$ (of $\sigma$); if the blowing-up is admissible, it is not difficult to see that $\widetilde{E}$ is a SNC divisor. From now on, we denote an admissible blowing-up by $\sigma: (\widetilde{M},\widetilde{E}) \to (M,E)$.

\paragraph{Sequence of admissible blowings-up.} A \emph{finite sequence of admissible blowings-up} is given by
\[
(M_r,E_r) \xrightarrow{\sigma_r} \ldots  \xrightarrow{\sigma_2} (M_1,E_1)\xrightarrow{\sigma_1} (M_0,E_0)
\]
where each successive blowing-up $\sigma_{i+1}: (M_{i+1},E_{i+1} )\to (M_i,E_i)$ is admissible (in respect to the exceptional divisor $E_i$). More generally, we abuse notation and we consider:

\begin{definition}\label{def:SeqAdmBlowings-up}
A \emph{sequence of admissible blowings-up} is a proper analytic morphism
\[
\sigma: (\widetilde{M},\widetilde{E}) \to (M,E)
\]
which is \emph{locally} a finite composition of admissible blowings-up. In other words, for each relatively compact $U \subset M$, the restricted morphism $\sigma|_{\sigma^{-1}(U)}$ is a finite sequence of admissible blowings-up.
\end{definition}

\subsection{Resolution of singularities of an analytic hypersurface}

Let $M$ be an analytic manifold and $E$ a SNC divisor. Consider an analytic (space) hypersurface $\Sigma \subset M$, and denote by $\mathcal{I}$ the (principal) reduced and coherent ideal sheaf whose zero set is $\Sigma$. The \emph{singular set} of $\Sigma$ (as an analytic space), which we denote by $Sing(\Sigma)$, is the set of points $\bar{z}\in M$ over which $\mathcal{I}$ has order at least two.

Given an admissible blowing-up $\sigma: (\widetilde{M},\widetilde{E}) \to (M,E)$ with center $\mathcal{C} \subset M$ and exceptional divisor $F$, we denote by $\sigma^{\ast}(\mathcal{I})$ the total transform of $\mathcal{I}$ (that is, the ideal sheaf which is generated by germs $f \circ \sigma$, where $f$ is a germ belonging to $\mathcal{I}$). We define the \emph{strict transform} $\widetilde{\mathcal{I}}$ of $\mathcal{I}$ by
\[
\widetilde{\mathcal{I}} =  \mathcal{I}_F^{-r} \cdot \sigma^{\ast}(\mathcal{I}),
\]
where $r$ is the maximal natural number such that $\widetilde{\mathcal{I}}$ is well-defined. The strict transform $\widetilde{\Sigma}$ of $\Sigma$ is the zero set of $\widetilde{\mathcal{I}}$. Note that we can extend the definition of strict transform to sequences of admissible blowings-up in a trivial way. 

Roughly, a resolution of singularities of $\Sigma$ is a sequence of admissible blowings-up $\sigma:(\widetilde{M},\widetilde{E}) \to (M,E)$, which is an isomorphism outside of $Sing(\Sigma)$, such that $\widetilde{\Sigma} \subset \widetilde{M}$ is an analytic smooth hypersurface which is transverse to the divisor $\widetilde{E}$: this means that, at every point $x\in \widetilde{E}$, there exists a coordinate system $(x_1,\ldots, x_n)$ centered at $x$ so that locally $\widetilde{E} =\{x_1 \cdots x_l =0\}$ and $\widetilde{\Sigma} = \{x_n=0\}$. This last condition guarantees that $\widetilde{E}|_{\widetilde{\Sigma}}$ is also a SNC divisor, and we may consider the pair $(\widetilde{\Sigma},\widetilde{E}|_{\widetilde{\Sigma}})$.



The classical Theorem of Hironaka adapted to the real-analytic setting (see e.g. \cite{bm97,w05}) and specialized to hypersurfaces, yields the following enunciate:

\begin{theorem}[Resolution of Singularities]\label{thm:RShypersurfaces}
Let $M$ be a real-analytic manifold, $E$ be a SNC crossing divisor over $M$. and $\Sigma \subset M$ be a reduced and coherent analytic (space) hypersurface of $M$. Then there exists a resolution of singularities of $\Sigma$. In other words, there exists a proper analytic morphism
\[
\sigma: (\widetilde{M},\widetilde{E}) \to (M,E)
\]
which is a sequence of admissible blowings-up (see Definition \ref{def:SeqAdmBlowings-up}) such that the strict transform of $\widetilde{\Sigma}$ of $\Sigma$ is smooth and transverse to $\widetilde{E}$ and the restriction of $\sigma$ to $\widetilde{M}\setminus \widetilde{E}$ is diffeomorphism onto its image $M\setminus Sing(\Sigma)$. 
\end{theorem}

\subsection{Log-rank and Hsiang-Pati coordinates}\label{app:HP}

In the proof of Lemma \ref{lem:lengthSaddle}, it is important to control the pulled-back metric after resolution of singularities. In order to do so, we have used the notion of \emph{Hsiang-Pati coordinates}, which follows from the original ideas of \cite{HPoriginal}, that we recall below.

We start by general considerations valid for any analytic map. Consider an analytic map $\pi : X \to Y$, where $X$ denotes a nonsingular real analytic surface (so, $X$ is 2-dimensional) with simple normal crossings divisor $E$, and $Y$ denotes a real-analytic manifold of dimension $N\geq 2$. Given a point $\bar{z}\in E$:
\begin{itemize}
\item We say that $\bar{z}$ is a $1$-point if there exists only one irreducible component of the divisor $E$ at $\bar{z}$. In this case, there exists a coordinate system $(u,v)$ centered at $\bar{z}$ such that $E = \{u=0\}$.
\item We say that $\bar{z}$ is a $2$-point if there exist two irreducible components of the divisor $E$ at $\bar{z}$. In this case, there exists a coordinate system $\pmb{u}= (u_1,u_2)$ centered at $\bar{z}$ such that $E = \{u_1\cdot u_2=0\}$.
\end{itemize}

For each point $\bar{z}\in X$, we define the \emph{logarithmic rank} of $\pi$ at $\bar{z}$ (in terms of a locally defined coordinate system at $\bar{z}$ and $\bar{w}=\pi(\bar{z})$) by
\[
\begin{aligned}
\mbox{log rk}_{\bar{z}}\pi &= \mbox{rk}_{\bar{z}}(\mbox{Jac}(\pi)) & & \qquad \text{if }\bar{z} \notin E\\
\mbox{log rk}_{\bar{z}}\pi &= \mbox{rk}_{\bar{z}} (\mbox{log Jac}(\pi)) = \mbox{rk}\left(
\begin{matrix}
u \displaystyle{\frac{\partial \pi_1}{\partial u}}  & \ldots & u \displaystyle{\frac{\partial \pi_N}{\partial u}} \\
\displaystyle{\frac{\partial \pi_1}{\partial v}}   & \ldots & \displaystyle{\frac{\partial \pi_N}{\partial v}}
\end{matrix}\right)  & & \qquad \text{if $\bar{z}$ is a $1$-point}\\
\mbox{log rk}_{\bar{z}}\pi &= \mbox{rk}_{\bar{z}} (\mbox{log Jac}(\pi)) = \mbox{rk} \left(
\begin{matrix}
u_1 \displaystyle{\frac{\partial \pi_1}{\partial u_1}}  & \ldots & u_1 \displaystyle{\frac{\partial \pi_N}{\partial u_1}} \\
u_2 \displaystyle{\frac{\partial \pi_1}{\partial u_2}}  & \ldots & u_2 \displaystyle{\frac{\partial \pi_N}{\partial u_2}}
\end{matrix}\right) & & \qquad \text{if $\bar{z}$ is a $2$-point}
\end{aligned}
\]

For more detail about the logarithmic rank, we refer to \cite[Section 2.1]{HP}. 

\begin{remark}\label{rk:LogRank}
Let $\Sigma_0$ denote the set of points $\bar{z} \in X$ such that $\mbox{log rk}_{\bar{z}}\pi =0$. If $\pi:X \to Y$ is a proper map, then $\pi(\Sigma_0)$ is a locally finite set of points in $Y$ (c.f. \cite[Section 3.2]{HP}).
\end{remark}

We say that $\pi:X \to Y$ has \emph{Hsiang-Pati coordinates} (with respect to $E$) if $\pi$ has maximal rank outside of $E$ and, for every point $\bar{z} \in E$ there exist coordinate systems $(u,v)$ (respectively $\pmb{u}=(u_1,u_2)$ if $\bar z$ is a 2-point) centered at $\bar{z}$, and $(\pi_1,\ldots, \pi_N)$ centered at $\bar{w}=\pi(\bar{z})$, such that:
\begin{align}
\pi_1 = v,& \quad \pi_2= u^{\beta}, && \quad \pi_i = h_i(u,v), && \quad \text{if $\mbox{log rk}_{\bar{z}}\pi=1$,}\label{eq:HP1}\\
\pi_1 = u^{\alpha},& \quad \pi_2= g_2(u) + u^{\beta}v,&& \quad \pi_i = g_i(u)+h_i(u,v), && \quad \text{if $\mbox{log rk}_{\bar{z}}\pi=0$,}\label{eq:HP11}\\
\pi_1 = \pmb{u}^{\alpha},& \quad \pi_2= g_2(\pmb{u}) + \pmb{u}^{\beta},&& \quad \pi_i = g_i(\pmb{u}) +h_i(\pmb{u}), && \quad \text{if $\bar{z}$ is a $2$-point,}
\label{eq:HP12}
\end{align}
where $d\pi_1\wedge d g_i(\cdot) \equiv 0$, $d\pi_1 \wedge d\pi_2 \not\equiv 0$, $\pi_1$ divides $g_i$, $u^{\beta}$ (or $\pmb{u}^{\beta}$) divides $h_i$, for each $i=2,\ldots, N$, and $u^{\alpha}$ divides $u^{\beta}$ (respectively $\pmb{u}^{\alpha}$ divides $\pmb{u}^{\beta}$). We now recall the main result of \cite{HP} (which strenghten \cite{HPoriginal}), specialized to the simpler case that $\mbox{dim} \,X =2$ and $\mbox{dim}\, Y = 3$.

\begin{theorem}[Hsiang-Pati coordinates]\label{prop:HP}
With the notation of Theorem \ref{thm:RShypersurfaces}, suppose in addition that $\mbox{dim} \,M=3$ (and, therefore, $\Sigma$ is a surface). Then, up to composing with further blowings-up, we can suppose that the resolution of singularities $\pi = \sigma|_{\widetilde{\Sigma}}: (\widetilde{\Sigma},\widetilde{E}|_{\widetilde{\Sigma}}) \to (\Sigma,Sing(\Sigma))$ has Hsiang-Pati coordinates (with  respect to $E$). Furthermore, the property of Hsiang-Pati coordinates is preserved by composing $\pi$ with a finite sequence of blowing-up of one of the following two forms:
\begin{itemize}
\item[(i)] A blowing-up with center $\bar{z}$, where $\mbox{log rk}_{\bar{z}}\pi =0$.
\item[(ii)] The principalization (of the pullback) of the maximal ideal $m_{\bar{w}}$, where $\bar{w} \in \Sigma$.
\end{itemize}
\end{theorem}
\begin{proof}[Proof of Theorem \ref{prop:HP}]
The existence of the resolution of singularities $\pi$ is guaranteed by \cite[Corollary 3.8]{HP} and \cite[Lemma 3.1]{HP}. The Hsiang-Pati property is preserved by $(i)$ either by direct computation or by \cite[Lemma 2.3(2), Remark 3.5(2), and Lemma 3.1]{HP}; and by $(ii)$ either by direct computation, or by \cite[Lemma 3.4, Lemma 2.3(2), and Lemma 3.1]{HP}.
\end{proof}

In this paper, we use the following consequence, which is a variation of \cite[Lemma 3.1]{HPoriginal}:

\begin{lemma}\label{lem:HPmetric}
With the notation of Theorem \ref{prop:HP}, let $g$ be a Riemmanian metric over $M$. Fix a point $\bar{z}$ in $\widetilde{\Sigma}$. Suppose that $\bar{z}$ is a $2$-point and that the expression of $\pi$ satisfies equation \eqref{eq:HP12}. Then, in a neighborhood of $\bar{z}$, the pulled-back metric $g^{\ast}$ of $g$ is bi-Lipschitz equivalent to the metric
\[
ds^2 = (d \pmb{u}^{\alpha})^2 + (d \pmb{u}^{\beta})^2.
\]
\end{lemma}
\begin{proof}[Proof of Lemma \ref{lem:HPmetric}]
Let $\bar{w}=\pi(\bar{z})$. We start by noting that locally (over $\bar{w}$) the metric $g$ is bi-Lipschitz equivalent to the Euclidean metric $Euc = d\pi_1^2 + d\pi_2^2 + d\pi_3^2$. Therefore, it is enough to prove that the pull-back of $Euc$ is bi-Lipschitz equivalent to $ds^2$. By equation \eqref{eq:HP12} we get
\[
\begin{aligned}
Euc^{\ast} &=  d\pi_1^2 + d\pi_2^2 + d\pi_3^2\\
&= (d \pmb{u}^{\alpha})^2 + [d(g_2(\pmb{u}) + \pmb{u}^{\beta})]^2 +  [d(g_3(\pmb{u}) + h(\pmb{u}))]^2 .
\end{aligned}
\]
Now, since $dg_i \wedge d\pi_1 \equiv 0$ and $\pi_1$ divides $g_i$, we see that
\[
d g_i(\pmb{u}) = \widetilde{g_i}(\pmb{u}) d\pmb{u}^{\alpha}
\]
for some analytic functions $\widetilde{g_i}$ for $i=2,3$. Furthermore, since $\alpha$ and $\beta$ are $\mathbb{Q}$-linearly independent (because $d\pi_1\wedge d\pi_2 \not\equiv0$ and $d\pi_1 \wedge d g_2 \equiv 0$) and $h(\pmb{u})$ is an analytic function divisible by $\pmb{u}^{\beta}$, we deduce that
\[
\begin{aligned}
d h(\pmb{u}) &= h_{\alpha}(\pmb{u})d\pmb{u}^{\alpha} + h_{\beta}(\pmb{u}) d\pmb{u}^{\beta}
\end{aligned}
\]
for some analytic functions $h_{\alpha}$ and $h_{\beta}$. 
Indeed, 
since $\alpha$ and $\beta$ are $\mathbb{Q}$-linearly independent, 
$\frac {d u_1}{u_1}$, $\frac {d u_2}{u_2}$ are 
$\mathbb{Q}$-linear 
combinations of  $\frac {d\pmb{u}^{\alpha}} {\pmb{u}^{\alpha}}$, $\frac 
{d\pmb{u}^{\beta}} {\pmb{u}^{\beta}}$.  Therefore, since $h(\pmb{u})$ is an analytic function divisible by $\pmb{u}^{\beta}$, 
\begin{align*}
\frac {d h(\pmb{u})} {\pmb{u}^{\beta}} &= h_1 \frac {d u_1}{u_1} + h_2 \frac {d u_2}{u_2} = \tilde h_\alpha\frac {d\pmb{u}^{\alpha}} {\pmb{u}^{\alpha}} + \tilde h_\beta
\frac {d\pmb{u}^{\beta}} {\pmb{u}^{\beta}}. 
\end{align*}
Then it suffices to multiply the above identity by $\pmb{u}^{\beta}$ 
(recall that $\pmb{u}^{\alpha}$ divides 
$\pmb{u}^{\beta}$).

This implies that
\[
\begin{aligned}
Euc^{\ast} &= \left(d\pmb{u}^{\alpha}\right)^2 \left(1 + h_{\alpha}(\pmb{u})^2 + \sum_{i=2}^3  \widetilde{g_i}(\pmb{u})^2   \right) +  \left(d\pmb{u}^{\beta}\right)^2\left( 1+ h_{\beta}(\pmb{u})^2\right)\\
&+ 2d\pmb{u}^{\alpha} \otimes d\pmb{u}^{\beta}\bigl(\widetilde{g}_2(\pmb{u})+ \left[\widetilde{g}_3(\pmb{u}) +h_{\alpha}(\pmb{u})\right] h_{\beta}(\pmb{u})  \bigr).
\end{aligned}
\]
On the one hand, by using the inequality $a^2+2ab+b^2 \leq 2(a^2+b^2)$, we get
\[
\begin{aligned}
Euc^{\ast} &\leq  2\left(d\pmb{u}^{\alpha}\right)^2 \left(1 + h_{\alpha}(\pmb{u})^2 + \sum_{i=2}^3  \widetilde{g_i}(\pmb{u})^2   \right) +  2\left(d\pmb{u}^{\beta}\right)^2\left( 1+ h_{\beta}(\pmb{u})^2\right)\\
&\leq K \left[ \left(d\pmb{u}^{\alpha}\right)^2 + \left(d\pmb{u}^{\beta}\right)^2\right]
\end{aligned}
\]
for some positive $K>0$. On the other hand, by using the inequality $a^2\geq 0$ we get
\[
\begin{aligned}
Euc^{\ast} &\geq  \left(1 + \widetilde{g_2}(\pmb{u})^2\right) \left(d\pmb{u}^{\alpha}\right)^2 +  \left(d\pmb{u}^{\beta}\right)^2 + 2\widetilde{g_2}(\pmb{u}) d\pmb{u}^{\alpha} \otimes d\pmb{u}^{\beta}.
\end{aligned}
\]
Since $|\widetilde{g}_2(\pmb{u})| \leq \ell$ for some $\ell>0$, it follows that there exists $k=k(\ell)>0$ small such that $(1-k)\sqrt{1+\widetilde{g}_2(\pmb{u})^2} \geq |\widetilde{g}_2(\pmb{u})|$.
Hence 
$$
2|\widetilde{g_2}(\pmb{u}) d\pmb{u}^{\alpha} \otimes d\pmb{u}^{\beta}| \leq (1-k)\left(1 + \widetilde{g_2}(\pmb{u})^2\right) \left(d\pmb{u}^{\alpha}\right)^2 +  (1-k)\left(d\pmb{u}^{\beta}\right)^2,
$$
from which  we deduce that
\[
Euc^{\ast} \geq  k\left[ \left(d\pmb{u}^{\alpha}\right)^2 + \left(d\pmb{u}^{\beta}\right)^2\right],
\]
concluding the proof of the lemma.
\end{proof}

\subsection{Reduction of singularities of a planar real-analytic line foliation}\label{app:ReductionLineFoliations}

Consider an analytic vector field $\mathcal{Z}$ over an open and connected set $U \subset \mathbb{R}^2$ and denote by $\eta$ the analytic $1$-form associated to it by the relation $i_{\mathcal{Z}}\omega_{U} = \eta$, where $\omega_{U}$ is the volume form associated to the Euclidean metric. A point $x\in U$ is said to be a singularity of $\mathcal{Z}$ if $\mathcal{Z}(x) =0$. We assume that $\mathcal{Z} \not\equiv 0$, which implies that the singular set $Sing(\mathcal{Z})$ is a proper analytic subset of $U$. We now recall the definition of elementary singularities (following \cite[Definition 4.27]{IY}):

\begin{definition}[Elementary singularities]\label{def:ElementarySing} Suppose that the origin $0 \in U$ is a singular point of $\mathcal{Z}$ and consider the Jacobian matrix $\mbox{Jac}(\mathcal{Z})$ associated to $\mathcal{Z}$. We say that $0$ is an elementary singularity of $\mathcal{Z}$ if $\mbox{Jac}(\mathcal{Z})$ evaluated at $0$ has at least one eigenvalue with non-zero real part.
\end{definition}

\begin{remark}[On elementary singularities]\label{rk:ElembentarySing}
Given a vector field $\mathcal{Z} = A(x,y) \partial_x + B(x,y) \partial_y$ defined in $\mathbb{R}^2$, the Jacobian of $\mathcal{Z}$ is given by
\[
\mbox{Jac}(\mathcal{Z})(x,y) = \left[\begin{matrix}
\partial_xA(x,y)& \partial_yA(x,y)\\ \partial_xB(x,y)& \partial_yB(x,y)
\end{matrix}\right].
\]
Therefore, the eigenvalues of $\mbox{Jac}(\mathcal{Z})$ are solutions (in $\lambda$) of the following polynomial equation:
\begin{equation}\label{eq:EigenvaluesJac}
\lambda^2 - \mbox{tr}(\mbox{Jac}(\mathcal{Z})) \lambda + \det(\mbox{Jac}(\mathcal{Z})) =0
\end{equation}
where $\mbox{tr}(\cdot)$ and $\det(\cdot)$ stand for the trace and determinant respectively. In particular:
\begin{itemize}
\item[(i)] if $\det(\mbox{Jac}(\mathcal{Z})(0))<0$, then the two solutions of equation \eqref{eq:EigenvaluesJac} are non-zero real numbers with different signs. This implies that $0$ is a saddle singularity of $\mathcal{Z}$.

\item[(ii)] if $\mathcal{Z}$ has an elementary singularity at $0$ and $\mbox{tr}(\mbox{Jac}(\mathcal{Z})(0))=0$, then $\det(\mbox{Jac}(\mathcal{Z})(0))<0$ (otherwise, the real part of the eigenvalues $\lambda$ would be zero, a contradiction).
\end{itemize}
\end{remark}

Given an analytic surface $\mathscr{S}$, we recall that a line foliation $\mathscr{L}$ is a coherent sub-sheaf of the tangent sheaf $T\mathscr{S}$ such that, for each point $x\in \mathscr{S}$, there exists a neighborhood $U_x$ of $x$ and a vector field $\mathcal{Z}$ defined over $U_x$ which generates $\mathscr{L}$. We define, therefore, the notion of elementary singularities of a line foliation in a trivial way. The objective of a reduction of singularities of a line foliation $\mathscr{L}$ is to provide a sequence of blowings-up so that the ``transform'' of $\mathscr{L}$ only have isolated elementary singularities (and is ``adapted'' to the divisor $E$). 

More precisely, consider an admissible blowing-up $\sigma: (\widetilde{\mathscr{S}},\widetilde{E}) \to (\mathscr{S},E)$ with center $\mathcal{C}$ and exceptional divisor $F$. The \emph{strict transform} of $\mathscr{L}$ is the analytic line foliation $\widetilde{\mathscr{L}}$ which satisfies
\[
\widetilde{\mathscr{L}} \cdot \mathcal{I}_F^{r} = d\sigma^{\ast}(\mathscr{L}),
\]
where $\mathcal{I}_F$ is the reduced ideal sheaf whose zero set is $F$, $d\sigma^{\ast}(\mathscr{L})$ is the pull-back of $\mathscr{L}$ (which might have poles) and $r$ is the maximal number in $\{-1\} \cup \mathbb{N}$ so that $\widetilde{\mathscr{L}}$ is well-defined.

Finally, we say that a line foliation $\mathscr{L}$ is adapted to a SNC divisor $E$ if, for each irreducible component $F$ of $E$:\\
- either $\mathscr{L}$ is everywhere tangent to $F$ (in which case, we say that $\mathscr{L}$ is non-dicritical over $F$);\\
- or $\mathscr{L}$ is everywhere transverse to $F$  (in which case, we say that $\mathscr{L}$ is dicritical over $F$).

The classical Bendixson-Seidenberg Theorem (see e.g. \cite[Theorem 3.3]{ADL} or \cite[Theorem 8.14 and Section 8K]{IY} and references there-within) stated in the notation of this work, yields the following:

\begin{theorem}[Reduction of singularities of planar line foliations]\label{thm:RSVectorField}
Let $\mathscr{S}$ be a real-analytic surface, $E$ be a SNC crossing divisor over $\mathscr{S}$, and $\mathscr{L}$ be an analytic line foliation over $\mathscr{S}$ which is everywhere non-zero. Then there exists a proper analytic morphism
\[
\pi: (\widetilde{\mathscr{S}},\widetilde{E}) \to (\mathscr{S},E)
\]
which is a sequence of admissible blowings-up (see Definition \ref{def:SeqAdmBlowings-up}), such that the strict transform $\widetilde{\mathscr{L}}$ is adapted to $\widetilde{E}$, and its singular points are all isolated and elementary. 
\end{theorem}


\end{document}